\documentclass[reqno,twoside]{amsart}

\usepackage{amsfonts}
\usepackage{amsmath,amssymb,amsthm,amsthm}
\usepackage{mathtools}
\usepackage{dsfont}
\usepackage{etoolbox}

\makeatletter
\def\eqnarray{\stepcounter{equation}\let\@currentlabel=\theequation
\global\@eqnswtrue
\tabskip\@centering\let\\=\@eqncr
$$\halign to \displaywidth\bgroup\hfil\global\@eqcnt\z@
  $\displaystyle\tabskip\z@{##}$&\global\@eqcnt\@ne
  \hfil$\displaystyle{{}##{}}$\hfil
  &\global\@eqcnt\tw@ $\displaystyle{##}$\hfil
  \tabskip\@centering&\llap{##}\tabskip\z@\cr}

\def\endeqnarray{\@@eqncr\egroup
      \global\advance\c@equation\m@ne$$\global\@ignoretrue}

\def\@yeqncr{\@ifnextchar [{\@xeqncr}{\@xeqncr[1pt]}}
\makeatother

\DeclareMathAlphabet\gothic{U}{euf}{m}{n}

\newcommand{\gota}{\gothic{a}}
\newcommand{\gotb}{\gothic{b}}

\newcommand{\graph}{{\mathop{\rm graph}}}
\newcommand{\RRe}{\mathop{\rm Re}}

\newcounter{teller}

\newenvironment{tabel}{\begin{list}%
{\rm  (\alph{teller})\hfill}{\usecounter{teller} \leftmargin=1.1cm
\labelwidth=1.1cm \labelsep=0cm \parsep=0cm}
                      }{\end{list}}

\newcounter{tellerr}

\newenvironment{tabeleq}{\begin{list}%
{\rm  (\roman{tellerr})\hfill}{\usecounter{tellerr} \leftmargin=1.1cm
\labelwidth=1.1cm \labelsep=0cm \parsep=0cm}
                         }{\end{list}}

\newtheorem{theorem}{Theorem}[section]
\newtheorem{proposition}[theorem]{Proposition}
\newtheorem{lemma}[theorem]{Lemma}
\newtheorem{corollary}[theorem]{Corollary}

\theoremstyle{definition}
\newtheorem{example}[theorem]{Example}
\newtheorem{definition}[theorem]{Definition}
\newtheorem{remark}[theorem]{Remark}

\newcommand\emb{\stackrel{\mathclap{\normalfont d}}{\hookrightarrow}}

\newcommand{\loc}{{\rm loc}}

\newcommand{\one}{\mathds{1}}

\def\RR{\mathbb{R}}

\def\NN{\mathbb{N}}

\def\CC{\mathbb{C}}

\def\pOm{\partial\Omega}

\def\supp{\operatorname{supp}}

\numberwithin{equation}{section}


\begin{document}
\title[Fractional powers of sectorial operators]{Fractional powers of sectorial operators 
via the Dirichlet-to-Neumann operator}
\author{W. Arendt}
\address{Wolfgang Arendt, Institute of Applied Analysis, University of Ulm. Helmholtzstr. 18, D-89069 Ulm (Germany)} 
\email{wolfgang.arendt@uni-ulm.de}

\author{A.F.M. ter Elst}
\address{A.F.M, ter Elst, Department of Mathematics, University of Auckland. Private bag 92019. Auckland 1142 (New Zealand)}
\email{terelst@math.auckland.ac.nz }

\author{M. Warma}
\address{Mahamadi Warma, University of Puerto Rico (Rio Piedras Campus), College of Natural Sciences,
Department of Mathematics, PO Box 70377 San Juan PR
00936-8377 (USA)}
\email{mahamadi.warma1@upr.edu, mjwarma@gmail.com}

\begin{abstract}
In the very influential paper \cite{CS07} Caffarelli and Silvestre studied regularity of $(-\Delta)^s$, $0<s<1$, by identifying fractional powers with a certain Dirichlet-to-Neumann operator. Stinga and Torrea \cite{ST10} and Gal\'e, Miana and Stinga \cite{GMS13} gave several more abstract versions of this extension procedure. The purpose of this paper is to study precise regularity properties of the Dirichlet  and the Neumann problem in Hilbert spaces. Then the Dirichlet-to-Neumann operator becomes an isomorphism between interpolation spaces and its part in the underlying Hilbert space is exactly the fractional power. 
\end{abstract}

\begingroup
\makeatletter
\patchcmd{\@settitle}{\uppercasenonmath\@title}{\large}{}{}
\patchcmd{\@setauthors}{\MakeUppercase}{\large\sc}{}{}
\makeatother
\maketitle
\endgroup

\section{Introduction}

In the very influential article \cite{CS07} Caffarelli and Silvestre study the fractional powers $(-\Delta)^s$, $0<s<1$, on $\RR^N$ of the operator $-\Delta$ by identifying the operator $(-\Delta)^s$ with a Dirichlet-to-Neumann operator with respect to an extension to the upper half-plane.
Subsequently, such extensions have been studied in more abstract settings by Stinga and Torrea \cite{ST10} as well as by Gal\'e, Miana and Stinga \cite{GMS13}.
They obtain in particular a representation formula for the associated Dirichlet problem analogous to the Poisson formula.
We also refer to \cite{CaSt,NOS} and their references for the case of symmetric second-order elliptic operators in divergence form with smooth coefficients on bounded open sets in $\RR^N$ subject to zero Dirichlet  and Neumann boundary conditions on $\pOm$.

Our contribution goes in the same direction.
Instead of Banach spaces and generators of semigroups as in the papers \cite{GMS13,ST10} mentioned above, we concentrate on Hilbert spaces and sectorial operators.
This allows us to obtain precise regularity results and well-posedness of the Dirichlet  and the Neumann problem.
The Dirichlet-to-Neumann operator will be shown to be an isomorphism between two interpolation spaces and its part in the underlying Hilbert space is exactly the fractional power of the given sectorial operator.
In this way, we prove, may be for the first time, uniqueness of the extensions.

To be more specific, we consider a Hilbert space $H$ and a sectorial operator 
$A$ on $H$ which is defined by a continuous, coercive form 
$\mathcal E\colon V\times V\to\CC$, 
where $V$ is a Hilbert space continuously and densely embedded in $H$, that is, $V\emb H$. 
Thus we have the usual Gelfand triple $V\emb H\emb V'$, 
where $V'$ denotes the antidual of the space $V$. 
Moreover, $\langle \mathcal Au,v\rangle_{V',V}=\mathcal E(u,v)$ defines an operator 
$\mathcal A\in\mathcal L(V,V')$, where $\mathcal L(V,V')$ is the space of all linear 
and bounded operators from $V$ to $V'$. 
The part of $\mathcal A$ in $H$ is the operator $A$. 
Given $0<s<1$ we consider the Bessel kind of equation
\begin{align}\label{11}
-u''(t)-\frac{1-2s}{t}u'(t)+\mathcal Au(t)=0,
\quad t\in (0,\infty)
\end{align}
as in the papers mentioned above.
We identify a precise function space and call the functions in the function space which solve  \eqref{11} $s$-harmonic. 
Then we show that for each $x\in [H,V]_s$ there is a unique $s$-harmonic function $u$ satisfying $u(0)=x$. 
Moreover, this function has an $s$-normal derivative $y$ in $[H,V']_s=[V',H]_{1-s}$ 
(see \eqref{44-1} in Section \ref{sec4} for more details). 
Here and throughout the paper, for all $0<\theta<1$ we denote by 
$[H,V]_\theta$ and $[H,V']_\theta = [V',H]_{1-\theta}$ the complex interpolation spaces.
The corresponding Dirichlet-to-Neumann operator $\mathcal D_s$ which 
associates to $x\in [H,V]_s$ the $s$-normal derivative $y$ turns out to be an isomorphism 
from $[H,V]_s$ to $[H,V']_s$. 
The part of the operator $\mathcal D_s$ in $H$ is the multiple $c_s A^s$
of the fractional power $A^s$ of $A$, where $c_s$ is an explicit constant depending only on $s$. 
For the proof we use a new version of the Kato--Lions method to 
associate a generator of a holomorphic semigroup to a sesquilinear form 
as it was established in \cite{AtE}. 
We also use the same representation formula used in \cite{GMS13,ST10}. 
Our proofs, however, are self-contained, using merely a few results of interpolation theory. 

The rest of the paper is structured as follows. 
We start with a short motivation for the result and the methods, 
by considering the square root of a bounded operator (Section \ref{sec2}). 
In Section \ref{sec3} we put together some properties of the  mixed Sobolev spaces 
related to fractions.
The Dirichlet  and Neumann problem is studied in Section \ref{sec4}. 
The main result on the identification of the Dirichlet-to-Neumann map with the 
fractional power in the coercive case is obtained in Section \ref{sec5}. 
In Section \ref{sec6} we drop the condition that $\mathcal E$ is coercive and assume
merely that $\mathcal E$ is sectorial with vertex zero.

\section{Appetizer: the square root of a bounded operator}\label{sec2}

Let $H$ be a Hilbert space over $\CC$ and let $A\in\mathcal L(H):=\mathcal L(H,H)$ be 
{\bf coercive}, i.e.\ there exists an $\alpha \in (0,1]$ such that
\begin{align}\label{e-coe}
\RRe \langle Ax,x\rangle\ge \alpha\|x\|_H^2
\end{align}
for all $x \in H$.
 Then there exists a unique accretive operator, denoted by $A^{\frac 12}$,
such that $(A^{\frac 12})^2=A$, see e.g.\ \cite[Theorem V.3.35]{Kat}.
 This operator $A^{\frac 12}$ can be realized as a Dirichlet-to-Neumann operator in the following way.
We consider the Sobolev space
\begin{align*}
W^{1,2}((0,\infty);H):=\{u\in L^2((0,\infty);H);\; u'\in L^2((0,\infty);H)\}
\end{align*}
and we recall that 
\begin{align*}
W^{1,2}((0,\infty);H)\hookrightarrow C_0([0,\infty);H):=\{u\in C([0,\infty);H):\lim_{t\to\infty}\|u(t)\|_H=0\}.
\end{align*}
Then $W^{2,2}((0,\infty);H)\hookrightarrow C^1([0,\infty);H)$, where
\begin{align*}
W^{2,2}((0,\infty);H):=\{u\in W^{1,2}((0,\infty);H): u'\in W^{1,2}((0,\infty);H)\}.
\end{align*}

We have the following result.

\begin{proposition}
For each $x\in H$ there exists a unique $u\in W^{2,2}((0,\infty);H)$ such that
\begin{equation}\label{21}
\begin{cases}
-u''(t)+Au(t)=0, 
\quad t \in (0,\infty),\\[5pt]
u(0)=x.
\end{cases}
\end{equation}
\end{proposition}

\begin{proof}
Define the sesquilinear form $\gotb\colon W^{1,2}((0,\infty);H)\times W^{1,2}((0,\infty);H)\to\CC$ by
\begin{align*}
\gotb(u,v)=\int_0^\infty\Big(\langle u'(t),v'(t)\rangle_H+\langle Au(t),v(t)\rangle_H\Big)\,dt.
\end{align*}
Then $\gotb$ is continuous and from \eqref{e-coe} we have that
\begin{align}\label{e22}
\RRe \gotb(u,u)
\ge\int_0^\infty\left(\|u'(t)\|_H^2+\alpha\|u(t)\|_H^2\right)\,dt\ge \alpha \|u\|_{W^{1,2}((0,\infty);H)}^2.
\end{align}
So $\gotb$ is coercive.
Next we show existence.
Let $x\in H$.
There exists a function 
$\phi\in W^{2,2}((0,\infty);H)$ such that $\phi(0)=x$.
By the Lax--Milgram Lemma there exists a unique $w\in W_0^{1,2}((0,\infty);H)$ such that $\gotb(w,v)=\gotb(\phi,v)$ for all $v\in W_0^{1,2}((0,\infty);H)$, where
\begin{align*}
W_0^{1,2}((0,\infty);H):=\{v\in W^{1,2}((0,\infty);H): v(0)=0\}.
\end{align*}
 Let $u:=\phi-w$.
Then $u\in W^{1,2}((0,\infty);H)$ and
\[
\int_0^\infty \Big( \langle u'(t),v'(t)\rangle_H+\langle Au(t),v(t)\rangle_H \Big) \,dt=0
\]
for all $v\in W_0^{1,2}((0,\infty);H)$.
This implies that $u''=Au$ weakly.
Since $Au\in L^2((0,\infty);H)$, one has that $u\in W^{2,2}((0,\infty);H)$ and so $u$ is a solution of \eqref{21}.

To show uniqueness, let $u\in W^{2,2}((0,\infty);H)$ be a solution of \eqref{21} such that $u(0)=0$. 
Then \eqref{e22} gives
\begin{align*}
0=\int_0^\infty \Big(\langle-u''(t),u(t)\rangle_H+\langle Au(t),u(t)\rangle_H \Big)\,dt
\ge \alpha\|u\|_{W^{1,2}((0,\infty);H)}^2.
\end{align*}
Hence $u=0$.
\end{proof}

Now we define the {\bf Dirichlet-to-Neumann operator} $\mathcal D\colon H\to H$ as follows.
Let $x\in H$.
Let $u\in W^{2,2}((0,\infty);H)$ be the unique solution of \eqref{21}.
Define $\mathcal D x:=-u'(0)$.
Then the following result holds.

\begin{theorem}
We have that $\mathcal D=A^{\frac 12}$.
\end{theorem}

\begin{proof}
We first show that $\mathcal D^2=A$. 
Let $x\in H$ and $u\in W^{2,2}((0,\infty);H)$ be such that $u(0)=x$ and $-u''(t)+Au(t)=0$ 
for all $t\in (0,\infty)$.
Let $w:=u'$.
Then $-w''(t)+Aw(t)=0$ weakly for all $t\in (0,\infty)$.
This shows that $w\in W^{2,2}((0,\infty);H)$ and $w$ is a solution of \eqref{21} with $w(0)=u'(0)$. 
Then $-w'(0)=\mathcal Du'(0)=\mathcal D(-\mathcal Dx)$.
Moreover, $u\in W^{3,2}((0,\infty);H)$ and $Ax=Au(0)=u''(0)=w'(0)=\mathcal D^2x$.

Next we show that $\mathcal D$ is accretive. 
Let $x\in H$ and let $u$ be the unique solution of \eqref{21}. Then
\begin{eqnarray*}
\RRe\langle \mathcal Dx,x\rangle_H
&=&\RRe\langle -u'(0),u(0)\rangle_H
=\RRe\int_0^\infty\frac{d}{dt}\langle u'(t),u(t)\rangle_H\,dt\\
&=&\RRe\int_0^\infty \Big( \langle u'(t),u'(t)\rangle_H+\langle u''(t),u(t)\rangle_H \Big)\,dt\\
&=&\int_0^\infty\Big(\|u'(t)\|_H^2+\RRe\langle Au(t),u(t)\rangle_H\Big)\,dt\\
&\ge &0.
\end{eqnarray*}
Hence $\mathcal D$ is accretive.
\end{proof}

The crucial argument in the proof above is to differentiate the differential equation $-u''+Au=0$.
This is possible since the operator $A$ is bounded.
For unbounded operators different arguments are needed.
For fractional powers other than squares, weighted Sobolev spaces are needed.
They are introduced in the next section.

\section{Sobolev spaces}\label{sec3}

The Dirichlet and the Neumann problems we have in mind are well posed in mixed Sobolev spaces which are known from interpolation theory.
We give the definition, cite results we shall need and prove an integration by parts formula.

Let $X$ be a Hilbert space.
We will consider spaces of integrable functions on $(0,\infty)$ with values in $X$. 
Derivatives will be taken in the {\bf distributional sense}; i.e.\ using the elements of the 
scalar space $C_c^\infty((0,\infty))$ of all infinitely differentiable $\CC$-valued 
functions with compact support as test functions.
Here is the precise definition.

\begin{definition}\label{def-31}
\mbox{}
\begin{tabel}
\item 
Let $u,v\in L_{\loc}^1(X):=L_{\loc}^1((0,\infty);X)$.
We say that $v$ is the {\bf weak derivative} of $u$ if
\begin{align*}
-\int_0^\infty\varphi'(t)u(t)\;dt=\int_0^\infty\varphi(t)v(t)\,dt
\end{align*}
for all $\varphi\in C_c^\infty((0,\infty))$.
In that case we write $u':=v$.

\item Let $E$ be a subspace of $L_{\loc}^1(X)$ and let $u\in L_{\loc}^1(X)$. 
We say that $u' \in E$ if there exists a $v\in E$ such that $v$ is the weak derivative of $u$. 
\end{tabel}
\end{definition}

The weak derivative is unique if it exists and for all 
$u\in C^1((0,\infty);X)$ the weak and classical derivatives coincide.

Let $X$, $Y$ be Hilbert spaces such that $Y\emb X$.
This means that $Y$ is a dense subspace of $X$ and the injection of $Y$ into $X$ is continuous. 
Fix $0<s<1$.
We define the space
\begin{eqnarray*}
W_s(X,Y):=\{u\in L_{\loc}^1(Y): u'\in L_{\loc}^1(X), & \; & 
    \Big( t\mapsto t^su(t) \Big)\in L_2^\star(Y)\mbox{ and } 
\\*
& &  \Big( t\mapsto t^su'(t) \Big) \in L_2^\star(X)\}
,
\end{eqnarray*}
where for $Z=X$ or $Z=Y$,
\begin{align*}
L_2^\star(Z):=L^2\Big(Z,\frac{dt}{t}\Big)=L^2\left((0,\infty);Z,\frac{dt}{t}\right).
\end{align*}
In order to avoid clutter we write $t^s$ for the function
$t \mapsto t^s$.
It is clear that $W_s(X,Y)$ endowed with the norm
\begin{eqnarray*}
\|u\|_{W_s(X,Y)}
&=&\Big(\|t^su\|_{L_2^\star(Y)}^2+\|t^su'\|_{L_2^\star(X)}^2 \Big)^{\frac 12}\\
&=&\Big(\int_0^\infty\left(\|u(t)\|_Y^2+\|u'(t)\|_X^2\right)t^{2s-1}\,dt\Big)^{\frac 12}
\end{eqnarray*}
is a Banach space and it is even a Hilbert space.

 We quote the following result from \cite[Proposition 1.2.10]{Lu95}.

\begin{proposition}\label{prop-32}
Let $u\in W_s(X,Y)$. 
Then $u(0):=\lim_{t\downarrow 0}u(t)$ exists in the norm on $X$. 
Moreover, $u(0)\in [X,Y]_{1-s}$. 
The map $u\mapsto u(0)$ from $W_s(X,Y)$ into $[X,Y]_{1-s}$ is continuous and surjective. 
\end{proposition}

Recall that $[X,Y]_{\theta}$ is the complex interpolation space between $X$ and $Y$ for all 
$0<\theta<1$. 
Note that the complex interpolation space $[X,Y]_\theta$ coincides with the 
trace-method real interpolation space $(X,Y)_{\theta,2}$ since we restrict ourselves to 
Hilbert spaces.

Let 
\begin{eqnarray*}
C_c^\infty([0,\infty);Y):=\{u:[0,\infty)\to Y & : & u \mbox{ is infinitely differentiable } 
  \mbox{ and }  \\
& & \supp u \mbox{ is compact in } [0,\infty) \} .
\end{eqnarray*}
Clearly $C_c^\infty([0,\infty);Y)$ is a subspace of $W_s(X,Y)$.

We need the following density result.

\begin{proposition} \label{ps301}
Let $s \in (0,1)$.
Then the following assertions hold.
\begin{tabel}
\item \label{ps301-1}
If $s \geq \frac{1}{2}$, then the space $C_c^\infty([0,\infty);Y)$ is dense in $W_s(X,Y)$.
\item \label{ps301-2}
If $s < \frac{1}{2}$, then the space 
\begin{align*}
\{ u \in W_s(X,Y) \cap C^\infty((0,\infty);Y) : \supp u \mbox{ is a bounded set in } (0,\infty)\} 
\end{align*}
is dense in $W_s(X,Y)$.
\end{tabel}
\end{proposition}

The proof of Proposition~\ref{ps301} requires quite some preparation.

Let $Z$ be a Hilbert space and $\theta \in (0,1)$.
Define the space
\[
W^\theta(Z)
= \{ u \in L^1_\loc(Z) : t^\theta u \in L_2^*(Z) \} 
,  \]
with the norm $\|u\|_{W^\theta(Z)} = \|t^\theta u\|_{L_2^*(Z)}$.
Note that $W^\theta(Z) = L^2((0,\infty);Z,t^{2\theta - 1} \, dt)$.

\begin{lemma} \label{ls302}
Let $\theta \in (0,1)$.
Let $u \in W^\theta(Z)$ and $r \in (0,\infty)$.
Define $L_r u \colon (0,\infty) \to Z$ by
\[
(L_ru)(t) = u(r^{-1} t)
.  \]
Then $L_r u \in W^\theta(Z)$ and 
$\|L_r u\|_{W^\theta(Z)} = r^{\theta} \|u\|_{W^\theta(Z)}$.
Moreover, $\lim_{r \to 1} L_r u = u$ in $W^\theta(Z)$
for all $u \in W^\theta(Z)$.
\end{lemma}

\begin{proof}
Let $u \in W^\theta(Z)$ and $r \in (0,\infty)$.
Clearly $L_r u \in L^1_\loc(Z)$.
Moreover, 
\[
\|t^\theta L_r u\|_{L_2^*(Z)}^2
= \int_0^\infty \|t^\theta u(r^{-1} t)\|_Z^2 \, \frac{dt}{t}
= r^{2\theta} \int_0^\infty \|t^\theta u(t)\|_Z^2 \, \frac{dt}{t}
= r^{2\theta} \|u\|_{W^\theta(Z)}^2
.  \]
This proves the first two claims.

If $u$ is a step function, then it is easy to see that 
$\lim_{r \to 1} L_r u = u$ in $W^\theta(Z)$.
Since the step functions are dense in $L^2((0,\infty);Z,t^{2\theta - 1} \, dt)$
by \cite[Lemma~3.26(1)]{Alt}, the lemma follows.
\end{proof}

\begin{remark} \label{rs303}
The space $(0,\infty)$ with multiplication is a one-dimensional Lie group.
The Haar measure is $\frac{dt}{t}$ and the corresponding $L_2$-space 
is $L_2^*$.
Lemma~\ref{ls302} states that the vector valued left representation in $W^\theta(Z)$
is well-defined and is a continuous representation of the group
$(0,\infty)$ in $W^\theta(Z)$.
\end{remark}

For the remaining of this section fix
for all $n \in \NN$ a function $\rho_n \in C_c^\infty((0,\infty))$
such that $\rho_n \geq 0$, $\supp \rho_n \subset (1 - \frac{1}{2n}, 1 + \frac{1}{n})$
and $\lim_{n \to \infty} \int_0^\infty \rho_n(t) \, \frac{dt}{t} = 1$.
For all $\chi \in C_c^\infty((0,\infty))$, $\theta \in (0,1)$ and 
$u \in W^\theta(Z)$ define $\chi * u \colon (0,\infty) \to Z$ by 
\[
(\chi * u)(t)
= \int_0^\infty \chi(r) u(r^{-1} t) \, \frac{dr}{r}
.  \]
Clearly $\chi * u \in C^\infty((0,\infty);Z)$.

\begin{lemma} \label{ls304}
Let $\theta \in (0,1)$ and $u \in W^\theta(Z)$.
Then the following assertions hold.
\begin{tabel}
\item \label{ls304-1}
If $\chi \in C_c^\infty((0,\infty))$, then 
$\chi * u \in W^\theta(Z) \cap C^\infty((0,\infty);Z)$.
\item \label{ls304-2}
$\lim_{n \to \infty} \rho_n * u = u$ in $W^\theta(Z)$.
\end{tabel}
\end{lemma}

\begin{proof}
\ref{ls304-1}.
Let $t \in (0,\infty)$.
Then 
\[
t^\theta \|(\chi * u)(t)\|_Z
\leq \int_0^\infty r^\theta |\chi(r)| \, (r^{-1} t)^\theta \|u(r^{-1} t)\|_Z \, \frac{dr}{r}
.  \]
So 
\[
\Big( \int_0^\infty \|t^\theta (\chi * u)(t)\|_Z^2 \, \frac{dt}{t} \Big)^{\frac 12}
\leq \int_0^\infty t^\theta |\chi(t)| \, \frac{dt}{t}
 \cdot 
   \Big( \int_0^\infty \|t^\theta u(t)\|_Z^2 \, \frac{dt}{t} \Big)^{\frac 12}
.  \]
Therefore $\chi * u \in W^\theta(Z)$.

\ref{ls304-2}.
Let $n \in \NN$.
Set $\lambda_n = \int_0^\infty \rho_n(r) \, \frac{dr}{r}$.
Then 
\[
\rho_n * u - \lambda_n u
= \int_0^\infty \rho_n(r) (L_r u - u) \, \frac{dr}{r}
.  \]
So 
\[
\|\rho_n * u - u\|_{W^\theta(Z)}
\leq |1 - \lambda_n| \, \|u\|_{W^\theta(Z)} 
   + \int_0^\infty \rho_n(r) \|L_r u - u\|_{W^\theta(Z)} \, \frac{dr}{r}
.  \]
Then the statement follows from Lemma~\ref{ls302} together with the 
condition that $\lim_{n \to \infty} \lambda_n = 1$.
\end{proof}

As an immediately consequence we obtain the next proposition.

\begin{proposition} \label{ps305}
Let $\theta \in (0,1)$.
Then the space $W^\theta(Z) \cap C^\infty((0,\infty);Z)$ is dense 
in $W^\theta(Z)$.
\end{proposition}

Now we are able to prove Proposition~\ref{ps301}.

\begin{proof}[{\bf Proof of Proposition~\ref{ps301}}]
The proof is in several steps.

{\bf Step 1}.
Let $s \in (0,1)$.
We claim that the space $W_s(X,Y) \cap C^\infty((0,\infty);Y)$ is dense in $W_s(X,Y)$.  
Indeed, let $u \in W_s(X,Y)$.
Then $\rho_n * u \in C^\infty((0,\infty);Y)$ for all $n \in \NN$ and 
$\lim_{n \to \infty} \rho_n * u = u$ in $W^s(Y)$ by Lemma~\ref{ls304}\ref{ls304-2}.
For all $n \in \NN$ define $\psi_n \in C_c^\infty((0,\infty))$ by 
$\psi_n(r) = \frac{1}{r} \rho_n(r)$.
Then $\psi_n \geq 0$ and $\supp \psi_n \subset (1 - \frac{1}{2n}, 1 + \frac{1}{n})$.
Moreover, $\lim_{n \to \infty} \int_0^\infty \psi_n(r) \, \frac{dr}{r} = 1$.
Hence $\lim_{n \to \infty} (\rho_n * u)'
= \lim_{n \to \infty} \psi_n * (u')
= u'$ in $W^s(X)$ by Lemma~\ref{ls304}\ref{ls304-2}, 
this time applied with $\rho_n$ replaced by $\psi_n$.
So $\lim_{n \to \infty} \rho_n * u = u$ in $W_s(X,Y)$ and the claim is proved.

{\bf Step 2}.
Let $s \in (0,1)$.
We show that the space
$ \{ u \in W_s(X,Y) \cap C^\infty((0,\infty);Y) : \supp u \mbox{ is a bounded set in } (0,\infty)\} $
is dense in $W_s(X,Y)$.
 In fact, since $Y$ is continuously embedded into $X$ there exists a constant $c > 0$ 
such that $\|y\|_X \leq c \|y\|_Y$ for all $y \in Y$.
Let $u \in W_s(X,Y) \cap C^\infty((0,\infty);Y)$.
Let $\chi \in C^\infty([0,\infty))$ be such that 
$\one_{[0,1]} \leq \chi \leq \one_{[0,2]}$.
For all $n \in \NN$ define $\chi_n \colon [0,\infty) \to \RR$ by 
$\chi_n(t) = \chi(\frac{t}{n})$.
Moreover, define $u_n = \chi_n u$.
Then $u_n \in W_s(X,Y)$.
If $n \in \NN$, then 
\[
\|t^s (u - u_n)\|_{L_2^*(Y)}^2
= \int_0^\infty \Big| t^s (1 - \chi_n)(t) \|u(t)\|_Y \Big|^2 \, \frac{dt}{t}
\leq \int_n^\infty \|t^s u(t)\|_Y^2 \, \frac{dt}{t}
.  \]
So $\lim_{n \to \infty} \|t^s (u - u_n)\|_{L_2^*(Y)} = 0$.
Next, $u_n' = \chi_n' u + \chi_n u'$ for all $n \in \NN$.
It follows similarly that 
$\lim_{n \to \infty} \|t^s (u' - \chi_n u')\|_{L_2^*(X)} = 0$.
We shall show that $\lim_{n \to \infty} \|t^s \chi_n' u\|_{L_2^*(X)} = 0$.
Let $n \in \NN$.
Then
\begin{eqnarray*}
\|t^s \chi_n' u\|_{L_2^*(X)}^2
& = & \frac{1}{n^2} \int_0^\infty \Big| t^s \chi'(\frac{t}{n}) \|u(t)\|_X \Big|^2 \, \frac{dt}{t}  \\
& \leq & \frac{\|\chi'\|_\infty^2}{n^2}
   \int_0^\infty \|t^s u(t)\|_X^2 \, \frac{dt}{t}  \\
& \leq & \frac{c^2 \|\chi'\|_\infty^2}{n^2} \, \|t^s u\|_{L_2^*(Y)}^2
.
\end{eqnarray*}
So $\lim_{n \to \infty} \|t^s \chi_n' u\|_{L_2^*(X)} = 0$ and 
hence $\lim_{n \to \infty} u_n = u$ in $W_s(X,Y)$.
Then  Step~2 follows by an application of Step~1.

{\bf Step 3}.
We prove the two statements of Proposition~\ref{ps301}.

\ref{ps301-2}.
This is a special case of Step~2.

\ref{ps301-1}. 
Let $u \in W_s(X,Y) \cap C^\infty((0,\infty);Y)$ and suppose that 
$\supp u$ is a bounded set in $(0,\infty)$.
For all $n \in \NN$ define $u_n \colon (0,\infty) \to Y$ by 
$u_n(t) = u(t + \frac{1}{n})$.
Then $u_n \in C_c^\infty([0,\infty);Y)$.
Moreover, if $n \in \NN$, then 
\begin{eqnarray*}
\|t^s u_n\|_{L_2^*(Y)}^2
& = & \int_0^\infty t^{2s-1} \|u(t + \frac{1}{n})\|_Y^2 \, dt
= \int_{\frac{1}{n}}^\infty (t - \frac{1}{n})^{2s-1} \|u(t)\|_Y^2 \, dt  \\
& \leq & \int_{\frac{1}{n}}^\infty t^{2s-1} \|u(t)\|_Y^2 \, dt
\leq \|t^s u\|_{L_2^*(Y)}^2
\leq \|u\|_{W_s(X,Y)}^2
,  
\end{eqnarray*}
where we have used that $2s-1 \geq 0$ in the first inequality.
Similarly, $\|t^s u_n'\|_{L_2^*(X)} \leq \|u\|_{W_s(X,Y)}$ for all $n \in \NN$.
Hence the sequence $(u_n)_{n \in \NN}$ is bounded in $W_s(X,Y)$.
Therefore it has a subsequence which converges weakly in $W_s(X,Y)$.
So $u$ is in the weak closure of $C_c^\infty([0,\infty);Y)$ in $W_s(X,Y)$.

Together with Step~2 it follows that $C_c^\infty([0,\infty);Y)$ is weakly 
dense in $W_s(X,Y)$.
Since $C_c^\infty([0,\infty);Y)$ is convex, it is then also norm dense 
in $W_s(X,Y)$.
\end{proof}

Next, we want to specify our settings to {\bf Gelfand triples}; i.e.\ we consider two  Hilbert spaces $H$, $V$ such that $V\emb H$.
  Let $i$ be the inclusion from $V$ into $H$.
Then the dual map $i^*$ is a continuous map from $H'$ into $V'$, where $H'$ and $V'$ denote the antidual of $H$ and $V$, respectively.
Since $i$ has dense image, the map $i^*$ is injective.
Moreover, it also has a dense image.
By the Riesz representation theorem one can identify $H$ with $H'$.
We call $H$ the {\bf pivot space}.
Thus one has the chain
\begin{align*}
V \hookrightarrow  H \simeq H' \hookrightarrow V'
\end{align*}
which is known as a Gelfand triple. 
Therefore one has the following continuous and dense embeddings
\begin{align*}
V\emb H\emb V'.
\end{align*}

\begin{remark}
By the spectral theorem up to unitary equivalence one can assume that $H=L^2(\Gamma,\sigma)$ for some measure space $(\Gamma,\Sigma,\sigma)$, and $V=L^2(\Gamma,m \, d\sigma)$ for some measurable function $m\colon\Gamma\to [1,\infty)$.
Then $V'=L^2(\Gamma,\frac{d\sigma}{m})$ and
 the duality is given by
\begin{align*}
\langle f,g\rangle_{V',V}=\int_{\Gamma}f(x)\overline{g(x)}\,d\sigma(x)
\end{align*}
for all $f\in V'$ and $g\in V$. 
Thus $\langle f,g\rangle_{V',V}$ is written in terms of the 
measure $\sigma$ without weight. 
This is the reason for calling $H$ the pivot space. 
In this unitary equivalent situation the complex interpolation space becomes
\begin{align*}
[H,V]_s=L^2(\Gamma,m^s \, d\sigma)
\quad \mbox{and} \quad
[H,V']_s=L^2(\Gamma,m^{-s} \, d\sigma).
\end{align*}
In particular, $[H,V]_s'=[H,V']_s$ and we have the new Gelfand triple
\[
[H,V]_s\emb H\emb [H,V']_s,
\]
with again $H$ as pivot space.
\end{remark}

The following integration by parts formula will be crucial for us.

\begin{proposition}\label{prop-35}
Let $0<s<1$. 
Let $w\in W_s(V',H)$ and $v\in W_{1-s}(H,V)$. 
Then
$t \mapsto \langle w'(t),v(t)\rangle_{V',V}$ and 
$t \mapsto \langle w(t),v'(t)\rangle_H$ are elements of $L^1((0,\infty))$.
Moreover, 
\[
-\int_0^\infty \langle w'(t),v(t)\rangle_{V',V} \,dt
= \int_0^\infty \langle w(t),v'(t)\rangle_H \,dt+\langle w(0),v(0)\rangle_{[H,V']_s,[H,V]_s}.
\]
\end{proposition}

\begin{proof}
Let $w\in W_s(V',H)$ and $v\in W_{1-s}(H,V)$. 
By definition $t^s w' \in L_2^*(V')$ and 
$t^{1-s} v \in L_2^*(V)$.
So $t \mapsto \langle w'(t),v(t)\rangle_{V',V}$ is an element of $L^1((0,\infty))$.
Similarly, $t^s w \in L_2^*(H)$ and $t^{1-s} v' \in L_2^*(H)$.
Consequently $t \mapsto \langle w(t),v'(t)\rangle_H$ is an element of $L^1((0,\infty))$.

Together with Proposition~\ref{prop-32} it follows that the map 
\[
(w,v) \mapsto 
\int_0^\infty \Big( \langle w'(t),v(t)\rangle_{V',V} + \langle w(t),v'(t)\rangle_H \Big) \,dt
+ \langle w(0),v(0)\rangle_{[H,V']_s,[H,V]_s}
\]
is continuous from $W_s(V',H) \times W_{1-s}(H,V)$ into $\CC$.
Hence it suffices to show that 
\begin{equation}
\int_0^\infty \Big( \langle w'(t),v(t)\rangle_{V',V} + \langle w(t),v'(t)\rangle_H \Big) \,dt
+ \langle w(0),v(0)\rangle_{[H,V']_s,[H,V]_s}
= 0
\label{eprop-35;2}
\end{equation}
for all $(w,v)$ in a dense subset of $W_s(V',H) \times W_{1-s}(H,V)$.

Let $w \in W_s(V',H) \cap C^\infty((0,\infty);H)$, 
$v \in W_{1-s}(H,V) \cap C^\infty((0,\infty);H)$ and suppose that 
both $\supp w$ and $\supp v$ are bounded sets in $(0,\infty)$.
Then 
\begin{eqnarray}
\lefteqn{
\int_0^\infty \Big( \langle w'(t),v(t)\rangle_{V',V} + \langle w(t),v'(t)\rangle_H \Big) \,dt
} \hspace*{30mm}  \nonumber  \\*
& = & \lim_{\varepsilon \downarrow 0} 
   \int_\varepsilon^\infty \Big( 
     \langle w'(t),v(t)\rangle_{V',V} + \langle w(t),v'(t)\rangle_H 
                           \Big) \,dt  \nonumber  \\
& = & \lim_{\varepsilon \downarrow 0} 
   \int_\varepsilon^\infty \Big( \langle w'(t),v(t)\rangle_H 
         + \langle w(t),v'(t)\rangle_H \Big) \,dt  \nonumber  \\
& = & \lim_{\varepsilon \downarrow 0} 
   - \langle w(\varepsilon),v(\varepsilon)\rangle_H
.
\label{eprop-35;1}
\end{eqnarray}
We distinguish two cases.

{\bf Case 1}.
Suppose that $s \geq \frac{1}{2}$.  \\
Let $w \in C_c^\infty([0,\infty);H)$ and 
$v \in W_{1-s}(H,V) \cap C^\infty((0,\infty);H)$ be such that $\supp v$ is a bounded set in $(0,\infty)$.
Then $\lim_{\varepsilon \downarrow 0} v(\varepsilon) = v(0)$ in $H$ 
by Proposition~\ref{prop-32}.
So 
\[
\lim_{\varepsilon \downarrow 0} 
   \langle w(\varepsilon),v(\varepsilon)\rangle_H
= \langle w(0),v(0)\rangle_H
= \langle w(0),v(0)\rangle_{[H,V']_s,[H,V]_s}
.  \]
Hence (\ref{eprop-35;2}) is valid by using (\ref{eprop-35;1}).
Since $C_c^\infty([0,\infty);H)$ is dense in $W_s(V,H)$ by 
Proposition~\ref{ps301}\ref{ps301-2} and the space
\begin{align*}
\{ v \in W_{1-s}(H,V) \cap C^\infty((0,\infty);V) : \supp v
      \mbox{ is a bounded set in  } (0,\infty)\} 
\end{align*}
is dense in $W_{1-s}(H,V)$ by Proposition~\ref{ps301}\ref{ps301-1}, 
the proposition follows in this case.

{\bf Case 2}.
Suppose that $s < \frac{1}{2}$.  \\
Obviously $1 - s \geq \frac{1}{2}$, so now the space 
$C_c^\infty([0,\infty);V)$ is dense in $W_{1-s}(H,V)$ by 
Proposition~\ref{ps301}\ref{ps301-2}.
Let $v \in C_c^\infty([0,\infty);V)$ and 
$w \in W_s(V,H) \cap C^\infty((0,\infty);V)$ with $\supp w$ a bounded set in $(0,\infty)$.
Then Proposition~\ref{prop-32} implies that 
$\lim_{\varepsilon \downarrow 0} w(\varepsilon) = w(0)$ in~$V'$.
Also $\lim_{\varepsilon \downarrow 0} v(\varepsilon) = v(0)$ in $V$
since $v \in C_c^\infty([0,\infty);V)$.
So 
\begin{eqnarray*}
\lim_{\varepsilon \downarrow 0} 
   \langle w(\varepsilon),v(\varepsilon)\rangle_H
& = & \lim_{\varepsilon \downarrow 0} 
   \langle w(\varepsilon),v(\varepsilon)\rangle_{V',V}  \\
& = & \langle w(0),v(0)\rangle_{V',V}
= \langle w(0),v(0)\rangle_{[H,V']_s,[H,V]_s}
.  
\end{eqnarray*}
By \eqref{eprop-35;1} one deduces \eqref{eprop-35;2}
and the density of Proposition~\ref{ps301}
completes the proof in this case.
\end{proof}

\section{The Dirichlet  and Neumann problem}\label{sec4}

The aim of this section is to prove well-posedness and regularity of solutions of a Dirichlet  and a Neumann problem.

Let $V$, $H$ be Hilbert spaces such that $V\emb H$ and let $\mathcal E\colon V\times V\to \CC$ 
be a continuous and coercive sesquilinear form.
So there are constants $\mu,M > 0$ such that 
$|\mathcal E(u,v)|\le M\|u\|_V\|v\|_V$ and 
$\RRe \mathcal E(u,u)\ge \mu\|u\|_V^2$
for all $u,v \in V$.
Denote by $\mathcal A\in\mathcal L(V,V')$ the operator given by
\begin{align*}
\langle \mathcal A u,v\rangle_{V',V}=\mathcal E(u,v)
\end{align*} 
for all $u,v \in V$.
Throughout the remainder of the paper, we shall  use the notation $\mathcal E(u):=\mathcal E(u,u)$.
Let $0<s<1$ be fixed throughout this section.
We are interested in the equation
\begin{align}\label{41}
u''(t)+\frac{1-2s}{t}u'(t)-\mathcal Au(t)=0, \quad t\in (0,\infty).
\end{align}
We shall see in Theorem~\ref{theo-46} that the Sobolev space in the next definition
is the correct space for the well-posedness of the Dirichlet problem.

\begin{definition}
An {\bf $(\mathcal E,s)$-harmonic function} (or shortly {\bf $s$-harmonic function}) is a 
function $u\in W_{1-s}(H,V)$ such that $t^{1-2s} u' \in W_s(V',H)$ and 
\begin{align}\label{42}
-(t^{1-2s}u')'(t) +t^{1-2s}\mathcal Au(t)=0 \mbox{ in } V' \mbox{ for a.e.\ } t \in (0,\infty).
\end{align}
\end{definition}

Note that both functions $(t^{1-2s}u')'$ and $t^{1-2s}\mathcal Au$ 
are in $L^2(V',t^{2s}\frac{dt}{t})$ 
so that we actually obtain an identity in this space. 
Note also that  \eqref{42} is equivalent to \eqref{41}.

If $u$ is $s$-harmonic, then  Proposition \ref{prop-32} implies that
\[
u(0):=\lim_{t\downarrow 0}u(t)
\]
exists in $H$ and is an element of $[H,V]_s$. 
Similarly,
\begin{align}\label{44-1}
-\lim_{t\downarrow 0}t^{1-2s}u'(t)
\end{align}
exists in $V'$ and is an element of $[H,V']_s$. 
We consider this limit as an {\bf s-normal derivative}.
If $s=\frac 12$ then it equals $-u'(0)$.

In this section we are interested in the following two problems.
\begin{itemize}
\item Given $x\in [H,V]_s$, the {\bf Dirichlet problem} consists in finding an $s$-harmonic function $u$ such that $u(0)=x$.

\item Given $y\in [H,V']_s$, the {\bf Neumann problem} consists in finding an $s$-harmonic function $u$ such that $y=-\lim_{t\downarrow 0}t^{1-2s}u'(t)$.  
\end{itemize}
We will see that both problems are well-posed.

We define the sesquilinear form $\gotb_s\colon W_{1-s}(H,V)\times W_{1-s}(H,V)\to\CC$ by
\begin{align}\label{form-bs}
\gotb_s(u,v):=\int_0^\infty \Big( \langle u'(t),v'(t)\rangle_H+\mathcal E(u(t),v(t)) \Big) t^{2(1-s)}\,\frac{dt}{t}.
\end{align}
Then $\gotb_s$ is continuous and coercive.

\begin{lemma}\label{lem-44}
Let $u$ be $s$-harmonic.
Write $y:=-\lim_{t\downarrow 0}t^{1-2s}u'(t)$ in $V'$.
 Then
\begin{align}\label{44}
\gotb_s(u,v)=\langle y,v(0)\rangle_{[H,V']_s,[H,V]_s}
\end{align}
for all $v\in W_{1-s}(H,V)$.
In particular,
\begin{align}\label{45}
\gotb_s(u)=\langle y,u(0)\rangle_{[H,V']_s,[H,V]_s}.
\end{align}
\end{lemma}

\begin{proof}
Note that $u\in W_{1-s}(H,V)$ since $s$-harmonic.
Set $w:=t^{1-2s}u'$.
Then $w\in W_s(V',H)$.
Let $v\in W_{1-s}(H,V)$. 
Then Proposition \ref{prop-35} gives
\begin{align*}
\int_0^\infty\langle w'(t),v(t)\rangle_{V',V}\,dt
=-&\int_0^\infty\langle w(t),v'(t)\rangle_H\,dt+\langle y,v(0)\rangle_{[H,V']_s,[H,V]_s}\\
=-&\int_0^\infty\langle u'(t),v'(t)\rangle_Ht^{1-2s}\,dt+\langle y,v(0)\rangle_{[H,V']_s,[H,V]_s}.
\end{align*}
Since $w'(t)=t^{1-2s}\mathcal Au(t)$ in $V'$ for a.e.\ $t \in (0,\infty)$, it follows that 
\begin{align*}
\langle w'(t),v(t)\rangle_{V',V}=t^{1-2s}\mathcal E(u(t),v(t))
\end{align*}
for a.e.\ $t \in (0,\infty)$.
 This proves \eqref{44}.
\end{proof}

Conversely, we may use the form $\gotb_s$ to prove $s$-harmonicity using only a small space of test functions.

\begin{lemma}\label{lem-45}
Let $u\in W_{1-s}(H,V)$.
Assume $\gotb_s(u,v)=0$ for all $v\in C_c^\infty((0,\infty);V)$.
Then $u$ is $s$-harmonic.
\end{lemma}

\begin{proof}
Let $\varphi\in C_c^\infty((0,\infty))$.
For all $v\in V$ define $\tilde v\in C_c^\infty((0,\infty);V)$ by 
$\tilde v(t)= \overline{\varphi(t)}v$.
Then by assumption
\begin{eqnarray*}
0=\gotb_s(u,\tilde v)
&=&\int_0^\infty \Big( \langle u'(t),\tilde v'(t)\rangle_H+\mathcal E(u(t),\tilde v(t)) \Big) t^{1-2s}\,dt\\
&=&\int_0^\infty\langle u'(t),v\rangle_H\varphi'(t)t^{1-2s}\,dt
   +\int_0^\infty\langle\mathcal Au(t),v\rangle_{V',V}\varphi(t)t^{1-2s}\,dt  \\
&=&\int_0^\infty \varphi'(t) \langle w(t),v\rangle_{V',V} \,dt
   +\int_0^\infty \varphi(t) \langle t^{1-2s} \mathcal Au(t),v\rangle_{V',V} \,dt  ,
\end{eqnarray*}
where $w=t^{1-2s}u'$.
Since $v\in V$ is arbitrary, Definition \ref{def-31} implies that 
\[
- w'+t^{1-2s} \mathcal Au
=0
\] 
in $L_{\loc}^1(V')$. 
Hence 
\[
-(t^{1-2s}u')'(t) +t^{1-2s} \mathcal Au(t) = 0
\]
in $V'$ for almost every $t \in (0,\infty)$.
Because $u\in W_{1-s}(H,V)$, one has 
$t^{1-s} u \in L_2^*(V)$, so  
$t^{1-s}\mathcal Au\in L_2^\star(V')$. 
Hence $t^s(t^{1-2s}\mathcal Au)\in L_2^\star(V')$ 
and this implies that $t^sw'\in L_2^\star(V')$. 
In addition $t^s w = t^s t^{1-2s} u' = t^{1-s} u' \in L_2^*(H)$, since 
$u \in W_{1-s}(H,V)$.
Therefore $w\in W_s(V',H)$.
We proved that $u$ is $s$-harmonic.
\end{proof}

We can now prove well-posedness of the Dirichlet problem and the Neumann problem.

\begin{theorem}\label{theo-46}
The following assertions hold.
\begin{tabel}
\item \label{theo-46-1}
 {\em ({\bf Dirichlet Problem})}.
Let $x\in [H,V]_s$.
Then there exists a unique $s$-harmonic function $u$ such that $u(0)=x$.

\item \label{theo-46-2}
 {\em ({\bf Neumann Problem})}.
Let $y\in [H,V']_s$.
Then there exists a unique $s$-harmonic function $u$ such that $\lim_{t\downarrow 0}-t^{1-2s}u'(t)=y$.
\end{tabel}
\end{theorem}

\begin{proof}
\ref{theo-46-1}.
By Proposition \ref{prop-32} there exists a $\phi\in W_{1-s}(H,V)$ such that $\phi(0)=x$. 
Define $L \colon W_{1-s}^0(H,V) \to \CC$ by 
$Lv:=\gotb_s(\phi,v)$, where $W_{1-s}^0(H,V):=\{v\in W_{1-s}(H,V): v(0)=0\}$, 
which is a closed subspace of 
$W_{1-s}(H,V)$. 
Then $L$ is continuous and anti-linear.
Since the form $\gotb_s$ is coercive, there exists a unique $w\in W_{1-s}^0(H,V)$ such that 
$\gotb_s(w,v)=Lv$ for all $v\in W_{1-s}^0(H,V)$. 
Let $u:=\phi-w$.
Then $u\in W_{1-s}(H,V)$, $u(0)=\phi(0)=x$ and $\gotb_s(u,v)=0$ 
for all $v\in W_{1-s}^0(H,V)$.
It follows from Lemma \ref{lem-45} that $u$ is $s$-harmonic.
This proves existence.
Uniqueness follows from Lemma  \ref{lem-44}.

\ref{theo-46-2}.
Define $L \colon W_{1-s}(H,V) \to \CC$ by $Lv:=\langle y,v(0)\rangle_{[H,V']_s,[H,V]_s}$.
Then $L$ is continuous and anti-linear by Proposition \ref{prop-32}.
By the Lax--Milgram Lemma there exists a unique $u\in W_{1-s}(H,V)$ such that $\gotb_s(u,v)=Lv$ for all $v\in W_{1-s}(H,V)$.
In particular, $\gotb_s(u,v)=0$ for all $v\in C_c^\infty((0,\infty);V)$.
It follows from Lemma \ref{lem-45} that $u$ is $s$-harmonic.
Let $z:=-\lim_{t\downarrow 0}t^{1-2s}u'(t)$ in the sense of $V'$.
Then $z\in [H,V']_s$ by Proposition \ref{prop-32}.
From Lemma \ref{lem-44} we deduce that
\begin{align*}
\langle y,v(0)\rangle_{[H,V']_s,[H,V]_s}=\gotb_s(u,v)=\langle z,v(0)\rangle_{[H,V']_s,[H,V]_s}
\end{align*}
for all $v\in W_{1-s}(H,V)$. 
Hence $y=z$ by the surjectivity in Proposition \ref{prop-32}.
This shows that $u$ solves the Neumann problem.
Uniqueness follows also from Lemma~\ref{lem-44}.
\end{proof}

Theorem \ref{theo-46} and Proposition \ref{prop-32} allow us to define the Dirichlet-to-Neumann operator $\mathcal D_s$ in the following way.

\begin{definition}\label{def-47}
Define $\mathcal D_s \colon [H,V]_s \to [H,V']_s$ as follows.
Let $x \in [H,V]_s$. 
Let $u$ be the unique $s$-harmonic function satisfying $u(0)=x$. 
Then $\mathcal D_sx=y$, where $y=-\lim_{t\downarrow 0}t^{1-2s}u'(t)$ in $V'$.
We call $\mathcal D_s$ {\bf the Dirichlet-to-Neumann operator} 
(with respect to $s$ and $\mathcal E$). 
\end{definition}

\begin{proposition} \label{proposition46}
The operator $\mathcal D_s$ is an isomorphism from $[H,V]_s$ onto $[H,V']_s$.
\end{proposition}

\begin{proof}
It follows from Theorem \ref{theo-46} that $\mathcal D_s$ is linear and bijective. 
We show that $\mathcal D_s^{-1}$ is continuous. 
Let $y\in [H,V']_s$ and set $x:=\mathcal D_s^{-1}y$. 
Let $u$ be the $s$-harmonic function satisfying $u(0)=x$ and 
$-\lim_{t\downarrow 0}t^{1-2s}u'(t)=y$. 
Then $\gotb_s(u)=\langle y,u(0)\rangle_{[H,V']_s,[H,V]_s}$ by \eqref{45}.
By Proposition \ref{prop-32} there exists a constant $c>0$ such that 
$\|v(0)\|_{[H,V]_s}\le c\|v\|_{W_{1-s}(H,V)}$ for all $v\in W_s(H,V)$. 
Let $\mu\in (0,1]$ be a coercivity constant for $\mathcal E$.
Then
\begin{eqnarray*}
\mu\|u\|_{W_{1-s}(H,V)}^2 
&\le&\RRe \gotb_s(u)=\RRe \langle y,u(0)\rangle_{[H,V']_s,[H,V]_s}\\
&\le& \|y\|_{[H,V']_s}\|u(0)\|_{[H,V]_s}\le c\|y\|_{[H,V']_s}\|u\|_{W_{1-s}(H,V)}.
\end{eqnarray*}
Hence
\begin{align}\label{46}
\|u\|_{W_{1-s}(H,V)}\le c \mu^{-1} \|y\|_{[H,V']_s}.
\end{align}
Therefore
\begin{align*}
\|x\|_{[H,V]_s}
=\|u(0)\|_{[H,V]_s}\le c\|u\|_{W_{1-s}(H,V)}\le c^2 \mu^{-1} \|y\|_{[H,V']_s}.
\end{align*}
This shows that $\mathcal D_s^{-1}$ is continuous. 
Then also $\mathcal D_s$ is continuous, by the bounded inverse theorem.
\end{proof}

The next proposition combines several results of this section.

\begin{proposition}\label{prop-49}
The set
\begin{align*}
\mathcal H{ar_s}:=\{u\in W_{1-s}(H,V): u \;\mbox{ is s-harmonic}\}
\end{align*}
is a closed subspace of $W_{1-s}(H,V)$. 
We provide $\mathcal H{ar_s}$ with the induced norm of $W_{1-s}(H,V)$.
Then the mappings 
$u\mapsto u(0)$ from $\mathcal H{ar_s}$ into $[H,V]_s$ and 
$u\mapsto -\lim_{t\downarrow 0}t^{1-2s}u'(t)$ from $\mathcal H{ar_s}$ into $[H,V']_s$
 are both isomorphisms.
\end{proposition}

\begin{proof}
It follows from Lemmas \ref{lem-45} and \ref{lem-44} that $\mathcal Har_s$ is a  closed subspace. 
The surjectivity of both maps is proved in Theorem~\ref{theo-46}
and the injectivity in Lemma~\ref{lem-44}.
The continuity of the first mapping follows from Proposition \ref{prop-32}
and the continuity of the second follows from \eqref{46}. 
The continuity of the inverses is a consequence of the closed graph theorem.
\end{proof}

We conclude this section by specifying to the case $s=\frac 12$, which is much simpler.

\begin{proposition}\label{prop-48}
A function $u$ is $\frac 12$-harmonic if and only if we have that $u\in W^{2,2}((0,\infty);V')\cap L^2((0,\infty);V)$ and $-u''+\mathcal Au=0$ in $L^2((0,\infty);V')$.
\end{proposition}

\begin{proof}
By definition $u$ is $\frac 12$-harmonic if and only if $u \in W_{\frac 12}(H,V)$, 
$u'\in W_{\frac 12}(V',H)$ and $-u''(t) + \mathcal{A} u(t) = 0$ in $V'$ 
for a.e.\ $t \in (0,\infty)$.
The latter is equivalent to 
$u \in L^2((0,\infty);V)$, $u' \in L^2((0,\infty);H)$, 
$u'' \in L^2((0,\infty);V')$ and $-u''(t) + \mathcal{A} u(t) = 0$ in $V'$ 
for a.e.\ $t \in (0,\infty)$.
By \cite[Chapter 1, Proposition 2.2]{LM72} one has the inclusion 
$W^{2,2}((0,\infty);V')\cap L^2((0,\infty);V)\subset W^{1,2}((0,\infty);H)$. 
Then the proposition follows.
\end{proof}

Proposition \ref{prop-49} has the following form if $s = \frac{1}{2}$.

\begin{theorem}\label{theo-411}
The set
\begin{align*}
\mathcal H{ar_\frac 12}=\{u\in W^{1,2}((0,\infty);H)\cap L^2((0,\infty);V): u 
\mbox{ is $\frac{1}{2}$-harmonic}\}
\end{align*}
is a closed subspace of $W^{1,2}((0,\infty);H)\cap L^2((0,\infty);V)$.
Moreover, the mappings 
$u\mapsto u(0)$ from $\mathcal H{ar_\frac 12}$ into $[H,V]_{\frac 12}$ and 
$u\mapsto -u'(0)$ from $\mathcal H{ar_\frac 12}$ into $[H,V']_{\frac 12}$
 are both isomorphisms.
 \end{theorem}
 
Denote by $A$ the part of the operator $\mathcal A$ in $H$; i.e. 
\begin{align*}
D(A)=\{x\in V,\;\mathcal Ax\in H\} \quad \mbox{and} \quad Ax=\mathcal Ax. 
\end{align*}
Then $A$ is an $m$-sectorial operator with vertex $\gamma>0$ in the sense of 
Kato \cite[Subsection~V.3.10]{Kat}. 
Moreover, $A$ has bounded imaginary powers, see for example \cite[Corollary~7.1.8]{Haase}.
Hence $[H,D(A)]_\theta = D(A^\theta)$ for all $\theta \in (0,1)$
by \cite[Theorem~1.15.3]{Tri}.

Finally, we mention the following $W^{2,2}$-regularity.
 
 \begin{proposition}\label{prop-412}
 Let $u$ be $\frac{1}{2}$-harmonic. 
Then the following assertions are equivalent.
 \begin{tabeleq}
 \item \label{prop-412-1}
$u\in W^{2,2}((0,\infty);H)$;
 
 \item  \label{prop-412-2}
$u(0)\in [H,D(A)]_{\frac 34}=D(A^{\frac 34})$;
 
 \item \label{prop-412-3}
 $u'(0)\in [H,D(A)]_{\frac 14}=[H,V]_{\frac 12}$.
 \end{tabeleq}
 \end{proposition}
 
The proof of Proposition \ref{prop-412} is based on the following properties of traces 
\cite[Chapter 1, Theorems 3.1 and 3.2]{LM72}.

\begin{proposition}\label{prop-413}
Let $Y\emb X$, where $X,Y$ are Hilbert spaces. 
Then the mapping $u\mapsto (u(0),u'(0))$ maps
$W^{2,2}((0,\infty);X)\cap L^2((0,\infty);Y)$ into $[X,Y]_{\frac 34}\times [X,Y]_{\frac 14}$
and is surjective.
\end{proposition}

\begin{proof}[\bf Proof of Proposition \ref{prop-412}]
Suppose \ref{prop-412-1} is valid.
Since $u''=\mathcal Au$, it follows that $\mathcal Au(t) = u''(t) \in H$
and hence $u(t) \in D(A)$ for almost all $t > 0$.
So $u\in L^2((0,\infty);D(A))\cap W^{2,2}((0,\infty);H)$.
Then the implications \ref{prop-412-1} $\Rightarrow$ \ref{prop-412-2} and 
\ref{prop-412-1} $\Rightarrow$ \ref{prop-412-3} follow directly 
from Proposition \ref{prop-413}.

\ref{prop-412-2} $\Rightarrow$ \ref{prop-412-1}.
Our proof is based on the Dore--Venni Theorem. 
We consider the negative Dirichlet Laplacian $B^D$ on $L^2((0,\infty);H)$ given by
\begin{equation*}
D(B^D)=\{w\in W^{2,2}((0,\infty);H): w(0)=0\}
\quad \mbox{and} \quad
B^Dw=-w''.
\end{equation*}
This is a selfadjoint, positive operator.
In fact it is associated with the closed form
\begin{align*}
\gotb^D(w,v)=\int_0^\infty\langle w'(t),v'(t)\rangle_H\,dt
\quad \mbox{and} \quad
D(\gotb^D)=W_0^{1,2}((0,\infty);H).
\end{align*}
The other operator in $L^2((0,\infty);H)$ which we consider is 
the operator $\mathcal A_2$ with domain 
$D(\mathcal A_2)=L^2((0,\infty);D(A))$ given by $(\mathcal A_2 w)(t) = A (w(t))$. 
Then the operators $-B^D$ and $-\mathcal A_2$ generate bounded holomorphic 
$C_0$-semigroups on $L^2((0,\infty);H)$ which commute. 
Moreover, $\mathcal A_2$ is invertible. 
It follows  from a version of the Dore--Venni Theorem \cite[Theorem~2.1]{Dore--Venni}
(see also \cite[Theorem 8.4]{Pr93} and \cite[Corollary 4.7]{Mon})
that the operator $B^D+\mathcal A_2$ with usual domain 
$D(B^D+\mathcal A_2)=D(B^D)\cap D(\mathcal A_2)$ is invertible.

By assumption we have $u(0) \in [H,D(A)]_{\frac 34}$.
Then by Proposition \ref{prop-413} there exists a 
$\phi\in W^{2,2}((0,\infty);H)\cap L^2((0,\infty);D(A))$ such that $\phi(0)= u(0)$. 
Let $f:=-\phi''+ \mathcal A_2 \phi\in L^2((0,\infty);H)$. 
Since $B^D+\mathcal A_2$ is invertible 
there exists a $w\in D(B^D)\cap D(\mathcal A_2)$ such that $-w''+\mathcal A_2 w=f$. 
Let $\tilde u=\phi-w$. 
Then $\tilde u\in W^{2,2}((0,\infty);H)\cap L^2((0,\infty);D(A))$ and
$-\tilde u''+\mathcal A_2 \tilde u=0$.
Therefore Proposition~\ref{prop-48} 
implies that $\tilde u$ is $\frac 12$-harmonic.
Moreover, $\tilde u(0) = \phi(0) = u(0)$.
Then Theorem~\ref{theo-46}\ref{theo-46-1} gives
$u = \tilde u \in W^{2,2}((0,\infty);H)$.
This proves~\ref{prop-412-1}.

\ref{prop-412-3} $\Rightarrow$ \ref{prop-412-1}.
The proof is similar, but here we consider the negative Laplacian 
with Neumann boundary conditions $B^N$ on $L^2((0,\infty);H)$; that is 
\begin{equation*}
D(B^N)=\{w\in W^{2,2}((0,\infty);H): w'(0)=0\}
\quad \mbox{and} \quad
B^Nw=-w''.
\end{equation*} 
This operator is associated with the closed form $\gotb^N$ given by
\begin{align*}
\gotb^N(w,v)=\int_0^\infty\langle w'(t),v'(t)\rangle_H\,dt
\quad \mbox{and} \quad
D(\gotb^N)=W^{1,2}((0,\infty);H).
\end{align*}
Then again by the Dore--Venni Theorem  the operator $B^N+\mathcal A_2$  with 
usual domain 
$D(B^N+\mathcal A_2)=D(B^N)\cap D(\mathcal A_2)$ is invertible, where $\mathcal A_2$ is 
as above.

By assumption we have $u'(0) \in [H,D(A)]_{\frac 14}$.
Then by Proposition \ref{prop-413} there exists a 
$\phi\in W^{2,2}((0,\infty);H)\cap L^2((0,\infty);D(A))$ such that $\phi'(0)= u'(0)$. 
Let $f:=-\phi''+\mathcal A_2 \phi\in L^2((0,\infty);H)$. 
Since $B^D+\mathcal A_2$ is surjective 
there exists a $w\in D(B^N)\cap D(\mathcal A_2)$ such that $-w''+\mathcal A_2 w=f$. 
Let $\tilde u=\phi-w$.
Then $\tilde u\in W^{2,2}((0,\infty);H)\cap L^2((0,\infty);D(A))$ and 
$-\tilde u''+\mathcal A_2 \tilde u=0$.
Moreover, $\tilde u'(0)=\phi'(0)= u'(0)$. 
Thus $\tilde u$ is the $\frac 12$-harmonic function satisfying $\tilde u'(0)=u'(0)$. 
So $u = \tilde u \in W^{2,2}((0,\infty);H)$ and we have shown \ref{prop-412-1}.
\end{proof}

We notice that Property \ref{prop-412-1} in Proposition \ref{prop-412} is interesting. 
It implies that $u''(t) \in H$ and $\mathcal Au(t) \in H$ for almost all $t > 0$.
This is a kind of maximal regularity.

\section{The fractional powers via the D-t-N operator}\label{sec5}

We adopt the notation and assumptions of Section \ref{sec4}; 
that is $V$ and $H$ are Hilbert spaces, $V\emb H$, 
the sesquilinear form $\mathcal E\colon V\times V\to\CC$ is 
continuous, coercive and $\mathcal A\in\mathcal L(V,V')$ is given by 
$\langle \mathcal Au,v\rangle_{V',V}=\mathcal E(u,v)$. 
The number $s\in (0,1)$ is fixed and $\mathcal D_s\colon[H,V]_s\to [H,V']_s$ is the 
Dirichlet-to-Neumann operator (see Definition \ref{def-47}). 
Note that $[H,V]_s\emb [H,V']_s=[H,V]_s'$. 
Let $D_s$ be the part of the operator $\mathcal D_s$ in~$H$.  
So $D_s$ is the operator in $H$ given by
\begin{align*}
D(D_s):=\{x\in [H,V]_s: \mathcal D_sx\in H\}
\quad \mbox{and} \quad
D_sx=\mathcal D_sx.
\end{align*}
Therefore the graph of $D_s$ is given by 
\begin{eqnarray*}
\graph(D_s)
= \{ (x,y)\in H\times H & : & \mbox{there exists an $s$-harmonic map $u$ such that}  \\
& & u(0) = x \mbox{ and } y=-\lim_{t\downarrow 0}t^{1-2s}u'(t) \mbox{ in } V' \}
.
\end{eqnarray*}
Recall that $A$ is the part of the operator $\mathcal A$ in $H$.
Denote by $A^s$ the fractional power of $A$. 
Our main result of this paper is the following.
Define $c_s:=2^{1-2s}\frac{\Gamma(1-s)}{\Gamma(s)}$.

\begin{theorem}\label{theo-51}
One has $c_sA^s=D_s$.
\end{theorem}

We first prove that $D_s$ is $m$-sectorial.
For that we use the following result \cite[Theorem 2.1]{AtE}.

\begin{proposition}\label{prop-52}
Let $W$ be  a Hilbert space and let $\gotb\colon W\times W\to\CC$ be a continuous
sesquilinear form. 
Let $H$ be a Hilbert space and $j\colon W\to H$ be a continuous, 
linear map with dense image. 
Suppose there exist $\mu > 0$ and $\omega \in \RR$ such that 
\[
\mu \|u\|_W^2
\leq \RRe \gotb(u) + \omega \|j(u)\|_H^2
\]
for all $u \in W$.
Then there exists a unique $m$-sectorial operator $B$ on $H$ such that
\begin{eqnarray*}
\graph(B)
= \{ (x,y)\in H\times H & : & \mbox{there exists a $u \in W$ such that} \\
& & j(u) = x \mbox{ and } 
    \gotb(u,v)=\langle y,j(v)\rangle_H \mbox{ for all } v\in W \}
.
\end{eqnarray*}
\end{proposition}
We call the operator $B$ in Proposition~\ref{prop-52} the 
{\bf operator associated with the pair $(\gotb,j)$}.

We wish to apply Proposition~\ref{prop-52} with $W=W_{1-s}(H,V)$
and $\gotb = \gotb_s$ 
the sesquilinear form given by \eqref{form-bs}. 
Recall that 
$\gotb_s\colon W_{1-s}(H,V)\times W_{1-s}(H,V)\to\CC$ is given by
\[
\gotb_s(u,v):=\int_0^\infty \Big( \langle u'(t),v'(t)\rangle_H
       +\mathcal E(u(t),v(t)) \Big) t^{2(1-s)}\,\frac{dt}{t}
,  \]
the form $\gotb_s$ is continuous and coercive. 
Define $j\colon W_{1-s}(H,V)\to H$ by $j(u)=u(0)$. 
Note that also $j$ depends on $s$.
Then $j$ is linear, continuous with dense image (see Proposition \ref{prop-32}).

\begin{proposition}\label{prop-53}
The operator associated with $(\gotb_s,j)$ is $D_s$.
In particular, $D_s$ is $m$-sectorial.
\end{proposition}

\begin{proof}
Let $B$ be the operator associated with $(\gotb_s,j)$. 
Let $(x,y)\in\mbox{graph}(B)$. 
Then there exists a $u\in W_{1-s}(H,V)$ such that $u(0)=x$ and 
$\gotb_s(u,v)=\langle y,v(0)\rangle_H$ for all $v\in W_{1-s}(H,V)$. 
Choosing in particular $v\in C_c^\infty((0,\infty);V)$, Lemma \ref{lem-45} shows that 
$u$ is $s$-harmonic. 
Let $z:= \mathcal D_s x = -\lim_{t\downarrow 0}t^{1-2s}u'(t)$ in $V'$. 
Then  \eqref{44} gives
\begin{align*}
\langle y,v(0)\rangle_H=\gotb_s(u,v)=\langle z,v(0)\rangle_{[H,V']_s,[H,V]_s}
\end{align*}
 for all $v\in W_{1-s}(H,V)$. 
Consequently, $\langle y,w\rangle_H=\langle z,w\rangle_{[H,V']_s,[H,V]_s}$ for all 
$w\in [H,V]_s$ by Proposition~\ref{prop-32}. 
This implies that $z=y\in H$. 
Hence $x \in D(D_s)$ and $D_s x = B x$.
We have shown that $B\subset D_s$. 
 
 Conversely, let $x\in D(D_s)$.
Write $y = D_s x \in H$.
Let $u$ be $s$-harmonic such that $u(0)=x$.
Then 
$y=-\lim_{t\downarrow 0}t^{1-2s}u'(t)$ in $V'$. 
Let $v\in W_{1-s}(H,V)$.
Then  \eqref{44} gives
\begin{align*}
\gotb_s(u,v)=\langle y,v(0)\rangle_{[H,V']_s,[H,V]_s}=\langle y,v(0)\rangle_H ,
\end{align*}
where we have used that $y\in H$. 
Hence $x=u(0)\in D(B)$ and $Bx=y$.
This shows that $D_s\subset B$.
\end{proof}

Recall that the operator $-A$ generates a holomorphic $C_0$-semigroup 
$(e^{-tA})_{t\ge 0}$ on $H$. 
In particular, the mapping $t\mapsto e^{-tA}$ is in $C^\infty((0,\infty);\mathcal L(H))$ 
and even in $C^\infty((0,\infty);D(A^k))$ for all $k \in \NN$ if we provide $D(A^k)$ with the norm 
$\|x\|_{D(A^k)}=\|A^kx\|_H$.
Moreover, $\|e^{-tA}\|_{\mathcal L(H)}\le e^{-\mu t}$ for all $t > 0$, 
where $\mu>0$ is a coercivity constant of the form $\mathcal E$.

Define the function $\mathcal U\colon[0,\infty)\to\mathcal L(H)$ by
\begin{align}\label{51}
\mathcal U(t)=\frac{1}{\Gamma(s)}\int_0^\infty e^{-\frac{t^2}{4r}}r^s e^{-rA} \, \frac{dr}{r}.
\end{align}
Then
\[
\mathcal U\in C^\infty((0,\infty);\mathcal L(H))\cap C([0,\infty);\mathcal L(H)).
\]
The function $\mathcal U$ has the following properties.

\begin{proposition}\label{prop-54}
The following assertions hold.
\begin{tabel}
\item \label{prop-54-1}
If $t > 0$, then $\mathcal U(t)x \in D(A)$ for all $x \in H$ and 
\begin{align}\label{53}
\mathcal U''(t)+\frac{1-2s}{t}\mathcal U'(t)=A \, \mathcal U(t).
\end{align}

\item \label{prop-54-2} 
$\mathcal U(0)=A^{-s}$.

\item \label{prop-54-3} 
Let $x\in D(A^s)$.
Define $u \in C^\infty((0,\infty);H)$ by $u(t):=\mathcal U(t)A^sx$.
Then 
\begin{align*}
\lim_{t\downarrow 0}-t^{1-2s}u'(t)=c_sA^s x
\end{align*}
in $H$.

\item \label{prop-54-4} 
There exist $\delta \in (0,\mu)$ and $M\ge 0$ such that 
$\|\mathcal U(t)\|_{\mathcal L(H)}\le Me^{-\delta t}$ for all $t>0$.

\item \label{prop-54-5}
Let $x\in D(A^2)$.
Then $\mathcal U(\cdot)A^sx$ is $s$-harmonic.
\end{tabel}
\end{proposition}

\begin{proof}
\ref{prop-54-1}. 
Since $(e^{-t A})_{t > 0}$ is a holomorphic semigroup there 
exists a constant $c_1 > 0$ such that $\|Ae^{-tA}\|_{\mathcal L(H)}\le c_1t^{-1}$ for all $t \in (0,1]$.
If $t \in (1,\infty)$, then the semigroup property gives
$\|Ae^{-tA}\|_{\mathcal L(H)} 
\le \|Ae^{-A}\|_{\mathcal L(H)} \, \|e^{-(t-1)A}\|_{\mathcal L(H)}
\leq c_1 e^{-\mu(t-1)}$.
Hence $\mathcal U(t)H\subset D(A)$ for all $t>0$.
It is straightforward  to verify the identity \eqref{53}.

\ref{prop-54-2}.
This part follows from \cite[(3.56)]{ABHN11}.

\ref{prop-54-3}. 
Let $t > 0$.
One has 
\begin{align}\label{54}
u'(t)=-\frac{1}{2 \Gamma(s)}\int_0^\infty e^{-\frac{t^2}{4r}} t r^{s-1} e^{-rA}A^sx \, \frac{dr}{r}.
\end{align}
Substituting $\tau=\frac{r}{t^2}$ gives
\begin{align*}
-t^{1-2s}u'(t)
=\frac{1}{2\Gamma(s)}\int_0^\infty e^{-\frac{1}{4\tau}}\tau^{s-1} e^{-t^2\tau A}A^sx \, \frac{d\tau}{\tau}.
\end{align*}
Hence $\lim_{t\downarrow 0}-t^{1-2s}u'(t)=c_sA^sx$, since
\begin{align*}
\frac{1}{2\Gamma(s)}\int_0^\infty e^{-\frac{1}{4\tau}}\tau^{s-1} \, \frac{d\tau}{\tau}
=\frac{1}{2\Gamma(s)}\int_0^\infty e^{-r}r^{1-s}4^{1-s} \, \frac{dr}{r}
=2^{1-2s}\frac{\Gamma(1-s)}{\Gamma(s)}
= c_s.
\end{align*}

\ref{prop-54-4}. 
Using the fact that $\|e^{-tA}\|_{\mathcal L(H)}\le e^{-\mu t}$ for all $t\ge 0$, 
one obtains that
\begin{eqnarray*}
\|\mathcal U(t)\|_{\mathcal L(H)}
& \le &\frac{1}{\Gamma(s)} \int_0^te^{-\mu r}e^{-\frac{t^2}{4r}}r^{s-1} \, dr+
   \frac{1}{\Gamma(s)} \int_t^\infty e^{-\mu r}e^{-\frac{t^2}{4r}}r^{s-1} \, dr  \\
& \le & \frac{e^{-\frac{t^2}{4t}}}{\Gamma(s)} \int_0^te^{-\mu r}r^{s-1} \, dr+
  \frac{e^{-\frac{\mu}{2}t}}{\Gamma(s)} \int_t^\infty e^{-\frac{\mu}{2} r}r^{s-1} \, dr \\
& \le & \mu^{-s} e^{-\frac t4}+ (2 \mu^{-1})^s e^{-\frac{\mu}{2}t}
\end{eqnarray*}
for all $t>0$.

\ref{prop-54-5}.
Write $u = \mathcal U(\cdot)A^sx$.
First we shall prove that $t^{1-s}u \in L_2^\star(V)$.
Observe that $D(A)\hookrightarrow V$.
Hence there exists a constant $c>0$ such that $\|z\|_V\le c\|Az\|_H$ for all $z\in D(A)$.
Since $x\in D(A^2)$ and $A \, \mathcal U(t)A^sx=\mathcal U(t)A^{1+s}x$, it suffices to show that 
\begin{align*}
\int_0^\infty t^{2(1-s)}\|\mathcal U(t)z\|_H^2 \, \frac{dt}{t}<\infty 
\end{align*}
for all $z \in H$.
From part \ref{prop-54-4} we easily see that 
\begin{align*}
\int_1^\infty t^{2(1-s)}\|\mathcal U(t)z\|_H^2 \, \frac{dt}{t}<\infty
\quad \mbox{and} \quad 
\int_0^1 t^{2(1-s)}\|\mathcal U(t)z\|_H^2 \, \frac{dt}{t}<\infty .
\end{align*}
Hence $t^{1-s}u \in L_2^\star(V)$.

Secondly, we show that $t^{1-s} u' \in L_2^\star(H)$.
In order to prove this we first
show that $\int_0^1 t^{2(1-s)}\|u'(t)\|_H^2 \, \frac{dt}{t} <\infty$.
In fact, by Statement~\ref{prop-54-3} there exists a constant $c>0$ such that 
$ t^{1-2s}\|u'(t)\|_H\le c$ for all $t \in (0,1]$.
Hence 
\begin{align*}
\int_0^1 t^{2(1-s)}\|u'(t)\|_H^2 \, \frac{dt}{t}
\le c^2 \int_0^1 t^{2s} \, \frac{dt}{t} < \infty.
\end{align*}
It remains to show that
\begin{align}\label{55}
\int_1^\infty t^{2(1-s)} \|u'(t)\|_H^2\, \frac{dt}{t} <\infty.
\end{align}
We shall use \eqref{54}.
Let $t \geq 1$.
Observe that 
\begin{eqnarray}
\int_0^\infty e^{-\frac{t^2}{4r}} t r^{s-1} e^{-r\mu} \, \frac{dr}{r}
&\leq& e^{-\frac t8} \int_0^t e^{-\frac{t^2}{8r}}
       t r^{s-1} e^{-r\mu} \, \frac{dr}{r} 
   + e^{-\frac{t\mu}{2}} \int_t^\infty 
   e^{-\frac{t^2}{4r}} t r^{s-1} e^{-\frac{r\mu}{2}} \, \frac{dr}{r} \nonumber   \\
& \le & e^{-\frac t8} t \int_0^\infty e^{-\frac{1}{8r}}
    r^{s-2} e^{-r\mu} \, dr  
    +t e^{-\frac{t\mu}{2}} \int_1^\infty 
     r^{s-2} e^{-\frac{r\mu}{2}} \, dr  \nonumber   \\
& \leq & ce^{-\varepsilon t} \nonumber
\end{eqnarray}
for a suitable constant $c>0$, where $\varepsilon = \frac{1}{9} \wedge \frac{\mu}{4}$.
Hence it follows from \eqref{54} that 
\[
 \|u'(t)\|_H
\le \frac{c}{2 \Gamma(s)}\|A^sx\|_H e^{-\varepsilon t}
\]
and the estimate \eqref{55} is valid.
So $t^{1-s} u' \in L_2^*(H)$.
Therefore $u \in W_{1-s}(H,V)$.

Finally it follows from Statement~\ref{prop-54-1} that 
$\gotb_s(u,v) = 0$ for all $v \in C_c^\infty((0,\infty);V)$.
Hence $u$ is $s$-harmonic by Lemma~\ref{lem-45}.
\end{proof}

Now we are able to prove Theorem \ref{theo-51}.

\begin{proof}[\bf Proof of Theorem \ref{theo-51}]
Let $x\in D(A^2)$.
We shall show that $x\in D(D_s)$ and $D_sx=c_sA^sx$.
In fact, by Proposition \ref{prop-54}(e) the function 
$u(\cdot):=\mathcal U(\cdot)A^sx$ is $s$-harmonic and $u(0)=x$ by Proposition \ref{prop-54}(b).
Moreover, $c_sA^sx=-\lim_{t\downarrow 0}t^{1-2s}u'(t)$ by Proposition \ref{prop-54}(c).
Thus, by the definition of $D_s$, one has $x\in D(D_s)$ and $D_sx=c_sA^sx$.
Since $D(A^2)$ is a core of $D(A^s)$ and $D_s$ is closed 
(as the operator $D_s$ is $m$-sectorial by Proposition \ref{prop-53}), 
it follows that $c_sA^s\subset D_s$.
Because $c_s A^s$ is $m$-sectorial and $D_s$ is sectorial 
one concludes that $c_sA^s=D_s$.
\end{proof}

Theorem \ref{theo-51} has the following corollary.

\begin{corollary}
Let $u$ be $s$-harmonic, $x=u(0)$ and $y=-\lim_{t\downarrow 0}t^{1-2s}u'(t)$ in~$V'$.
Then $x\in D(A^s)$ if and only if $y\in H$.
\end{corollary}

Moreover Proposition~\ref{prop-54}(e) extends to the following representation formula.

\begin{corollary}\label{cor-57}
Let $u$ be $s$-harmonic.
Then
\begin{align}\label{56}
u(t)=\frac{1}{\Gamma(s)}\int_0^\infty e^{-\frac{t^2}{4r}}r^s A^se^{-rA}x \, \frac{dr}{r}
\end{align}
for all $t>0$, where $x = u(0)$.
In particular, $u\in C^\infty((0,\infty);H)$.
Stronger, $u\in C^\infty((0,\infty);D(A^k))$ for all $k\in \NN$.
\end{corollary}

\begin{proof}
Note that $x \in [H,V]_s$ by Proposition~\ref{prop-32}.
Since first $D(A^2)$ is a core for $A$, 
secondly the domain $D(A)$ with graph norm
is densely and continuously embedded in $V$ and 
thirdly $V$ is densely and continuously embedded in $[H,V]_s$, it follows 
that $D(A^2)$ is dense in $[H,V]_s$.
Hence there exists a sequence $(x_n)_{n \in \NN}$ in $D(A^2)$
such that $x_n \to x$ in $[H,V]_s$.
Now it follows from the second statement in Proposition \ref{prop-49} that $u_n\to u$ in 
$W_{1-s}(H,V)$, where $u_n$ is the $s$-harmonic function satisfying $u_n(0)=x_n$ 
for all $n \in \NN$.
Since $W_{1-s}(H,V)\subset C((0,\infty);H)$, 
the closed graph theorem implies that $u_n(t)\to u(t)$ in $H$ as $n\to\infty$ for all $t>0$.
We know from Proposition \ref{prop-54}\ref{prop-54-5} and \ref{prop-54-2} that
\begin{align*}
u_n(t)=\frac{1}{\Gamma(s)}\int_0^\infty e^{-\frac{t^2}{4r}}r^s e^{-rA}A^sx_n \, \frac{dr}{r}
\end{align*}
for all $t > 0$.
Note that there exists a constant $c > 0$ such that 
$\|A^se^{-rA} y\|_{H}\le \frac{c}{r^s} \|y\|_H$ for all $r \in (0,1]$
and $y \in H$.
Letting $n\to\infty$, we get \eqref{56} by Lebesgue's Theorem.
\end{proof}

As a consequence of these results, each $s$-harmonic function is a 
classical solution of the equation
\begin{align*}
u''(t)+\frac{1-2s}{t}u'(t)-Au(t)=0, \quad t \in (0,\infty).
\end{align*}

Now we want to rephrase the results for a concrete operator $A$.

\begin{example}[\bf Fractional power of the Dirichlet Laplacian]
Let $\Omega\subset\RR^N$ be a bounded, open set and let 
$A$ be the negative Dirichlet Laplacian on $L^2(\Omega)$, that is 
\begin{align*}
D(A)=\{u\in W_0^{1,2}(\Omega): \Delta u\in L^2(\Omega)\}
\quad \mbox{and} \quad
 Au=-\Delta u.
\end{align*}
Then $A$ is associated with the classical Dirichlet form 
$\mathcal E\colon W_0^{1,2}(\Omega)\times W_0^{1,2}(\Omega)\to \CC$ given by
\begin{align*}
\mathcal E(u,v)=\int_{\Omega}(\nabla u)\cdot(\overline{\nabla v}).
\end{align*}
It follows from Corollary \ref{cor-57} that each $s$-harmonic function can be 
identified with a function in $C^\infty((0,\infty)\times\Omega)$.
More precisely, let $u\in C^\infty((0,\infty)\times\Omega)$.
Then $u$ is $s$-harmonic if and only if 
\[
\begin{array}{l}
u(t,\cdot)\in W_0^{1,2}(\Omega) \mbox{ for all } t>0,  \\[5pt]
\displaystyle \int_0^\infty\int_\Omega 
   \Big( |\nabla_xu(t,x)|^2 + |\partial_t u(t,x)|^2 + |u(t,x)|^2
   \Big) \,dx \, t^{2(1-s)} \frac{dt}{t}  <\infty, \mbox{ and}  \\[5pt]
\displaystyle \partial_t^2 u(t,x) +\frac{1-2s}{t} \partial_tu
   +\Delta_x u(t,x)=0 \mbox{ for all } (t,x) \in (0,\infty)\times\Omega.
\end{array}
\]
Let $u$ be an $s$-harmonic function.
Then $u(0,\cdot):=\lim_{t\downarrow 0}u(t,\cdot)$ exists in 
$L^2(\Omega)$ and is an element of $[L^2(\Omega),W_0^{1,2}(\Omega)]_s$. 
Also $w:=-\lim_{t\downarrow 0} t^{1-2s} \partial_t u(t,\cdot)$ exists in 
$W^{-1,2}(\Omega):=(W_0^{1,2}(\Omega))'$ and is an element of 
$[L^2(\Omega), W^{-1,2}(\Omega)]_s$.
Conversely, for each $u_0\in [L^2(\Omega),W_0^{1,2}(\Omega)]_s$ there exists a unique 
$s$-harmonic function $u$ such that $u(0,\cdot)=u_0$.
One has $u_0\in D(A^s)$ if and only if  
$w:=-\lim_{t\downarrow 0}t^{1-2s} \partial_t u(t,\cdot)$ exists in $W^{-1,2}(\Omega)$
and $w \in L^2(\Omega)$.
In that case $c_sA^su_0=w$.
This is Theorem \ref{theo-51} rephrased for the Dirichlet-Laplacian.
\end{example}

Finally we specify the results for $s=\frac 12$.
Let $x\in [H,V]_{\frac 12}$.
By Theorem \ref{theo-46}(a) there is a unique $\frac 12$-harmonic function $u$ 
such that $u(0)=x$.
Then $y:=-u'(0)\in [H,V']_{\frac 12}$.
Then by definition $\mathcal D_{\frac 12}x=y$.
Moreover, $\mathcal D_{\frac 12}\colon[H,V]_{\frac 12}\to [H,V']_{\frac 12}$
is an isomorphism by Proposition~\ref{proposition46}.
Also $\mathcal D_{\frac 12}x\in H$ if and only if $x\in D(A^{\frac 12})$
by Theorem~\ref{theo-51}.
In that case $\mathcal D_{\frac 12}x=A^{\frac 12}x$, since $c_{\frac 12}=1$.

\begin{proposition}
Let $x\in D(A^{\frac 12})$.
Then the unique $\frac 12$-harmonic function $u$ satisfying $u(0)=x$ 
satisfies $u(t)=e^{-t A^{\frac 12}}x$ for all $t > 0$.
Hence
\begin{align}\label{58}
e^{-tA^{\frac 12}}x=\mathcal U(t)A^{\frac 12}x
\end{align}
for all $t > 0$.
\end{proposition}

\begin{proof}
There exist $\varepsilon>0$ and $M\ge 1$ such that 
$\|e^{-tA^{\frac 12}}\|_{\mathcal L(H)}\le Me^{-\varepsilon t}$ for all $t\ge 0$ 
(see for example \cite[Theorem 5.1.12]{ABHN11}).
Let $x\in D(A)$.
Define $u \colon (0,\infty) \to H$ by $u(t) = e^{-t A^{\frac 12}}x$.
Since $D(A)$ is continuously embedded into $V$, there exists a constant $c > 0$
such that $\|y\|_V \leq c \|A y\|_H$ for all $y \in D(A)$.
Then 
$\|u(t)\|_V 
\leq c \|e^{-tA^{\frac 12}}\|_{\mathcal L(H)} \|Ax\|_H
\leq c M e^{-\varepsilon t} \|Ax\|_H$
for all $t > 0$ and $u \in L_2((0,\infty);V)$.
Moreover, $u\in C^2([0,\infty);H)$ and $u''(t)=Au(t) = e^{-tA^{\frac 12}} Ax$
for all $t > 0$.
Hence $u \in W^{2,2}((0,\infty);H) \subset W^{2,2}((0,\infty);V')$
and $-u'' + \mathcal A u = 0$ in $L_2((0,\infty);V')$.
Then $u$ is $\frac{1}{2}$-harmonic by Proposition~\ref{prop-48}.
Now \eqref{58} follows from \eqref{56}.
Since $D(A)$ is dense in $D(A^{\frac 12})$ the identity \eqref{58} remains true for all 
$x\in D(A^{\frac 12})$.
\end{proof}

We conclude this section commenting on the integral representation \eqref{56}.

\begin{remark}
Let $\nu \in \RR$.
The {\bf Modified Bessel's Equation}
\begin{align*}
t^2w''(t)+tw'(t)-(t^2+\nu^2)w(t)=0, \quad (t>0)
\end{align*}
has the modified Bessel function of second kind $K_\nu$ as one of its solutions.
An integral representation for $K_\nu$ is given by
\begin{align}\label{511}
K_\nu(t)=\frac 12\left(\frac t2\right)^{\nu}\int_0^\infty e^{-r}e^{-\frac{t^2}{4r}}r^{-\nu} \, \frac{dr}{r}
\end{align}
for all $t>0$, see for example \cite[10.32.10]{DLMF}.
One has $K_\nu(t)\sim \sqrt{\frac{\pi}{2t}}e^{-t}$ as $t\to\infty$.
In our context $0<s<1$ is given.
Let $\lambda>0$.
Define $\psi \colon (0,\infty) \to (0,\infty)$ by
\begin{align}\label{512}
\psi(t)=\left(\sqrt{\lambda}t\right)^sK_s(\sqrt\lambda t).
\end{align}
Then 
\[
\psi''(t)+\frac{1-2s}{t}\psi'(t)-\lambda\psi(t)=0
\]
for all $t>0$ as a direct computation shows.
Using \eqref{511} and \eqref{512} one deduces that 
\begin{align}\label{512-2}
\psi(t)=\frac{1}{2^{1+s}}\lambda^st^{2s}\int_0^\infty 
      e^{-r}e^{-\frac{\lambda t^2}{4r}}r^{-s} \, \frac{dr}{r}
\end{align}
for all $t>0$.
Moreover, a substitution in \eqref{512-2} gives
\begin{align}\label{514}
\psi(t)=\lambda^s2^{s-1}\int_0^\infty e^{-\lambda r}e^{-\frac{t^2}{4r}}r^s\,\frac{dr}{r}
\end{align}
for all $t>0$.
Thus our approach \eqref{51} is a functional calculus which consists in replacing the parameter 
$\lambda$ by the operator $A$ in \eqref{514}.
Formula \eqref{514} is also used in \cite{GMS13}.
\end{remark}

\section{The non-coercive case}\label{sec6}

Up to now we used that the form $\mathcal E$ is coercive.
In this section we wish to replace this by the much weaker condition that 
$\mathcal E$ is merely sectorial with vertex~$0$.

In general, if $\gota \colon D(\gota) \times D(\gota) \to \CC$ is a 
sesquilinear form, then we say that $\gota$ is {\bf sectorial with vertex~$0$}
if there exists a $\theta \in [0,\frac{\pi}{2})$ such that $\gota(u) \in \Sigma_\theta$
for all $u \in D(\gota)$, where 
\[
\Sigma_\theta
= \{ r e^{i \alpha} : r \in [0,\infty) \mbox{ and } \alpha \in [-\theta,\theta] \} 
.  \]
In Theorem~\ref{ts601} we associate an $m$-sectorial operator to a densely defined sectorial 
form with vertex~$0$.
There is even a $j$-version of it like in Proposition~\ref{prop-52} that turns
out to be very useful in this section.

\begin{theorem} \label{ts601}
Let $H$ be a Hilbert space, $\gota \colon D(\gota) \times D(\gota) \to \CC$
a sectorial form with vertex~$0$ and $j \colon D(\gota) \to H$ a linear 
map with dense image.
Then there exists a unique $m$-sectorial operator $B$ in $H$ such that 
\begin{eqnarray*}
\graph(B) 
= \{ (x,y) \in H \times H & : & 
   \mbox{there exists a sequence $(u_n)_{n \in \NN}$ in $D(\gota)$ such that}  \\
& & (i) \; \lim_{n \to \infty} j(u_n) = x \mbox{ in } H,  \\
& & (ii) \; \sup \{ \RRe \gota(u_n) : n \in \NN \} < \infty, \mbox{ and} \\
& & (iii) \; \lim_{n \to \infty} \gota(u_n,v) = \langle y, j(v) \rangle_H
                \mbox{ for all } v \in D(\gota)
   \} .
\end{eqnarray*}
\end{theorem}
\begin{proof}
This is a special case of \cite[Theorem~3.2]{AtE}.
\end{proof}

Note that $B$ is the same operator as in Proposition~\ref{prop-52} if the domain
$D(\gota)$ is provided with a Hilbert space structure 
such that $j$ is continuous and the form $\gota$ is coercive and continuous.
We call the operator $B$ in Theorem~\ref{ts601} the {\bf operator associated with 
$(\gota,j)$}.
In particular, if $\gota$ is a densely defined sectorial form with vertex~$0$
in a Hilbert space $H$, then one can choose for $j$ the identity map
and we obtain an $m$-sectorial operator, which we call the 
{\bf operator associated with $\gota$}.

Now we extend the previous results for coercive forms to sectorial forms.

\begin{theorem} \label{ts602}
Let $H, V$ be Hilbert spaces such that $V\emb H$ and let 
$\mathcal E\colon V\times V\to\CC$ be a continuous sectorial form with vertex~$0$.
Let $A$ be the operator associated with~$\mathcal E$.
Further, let $s \in (0,1)$ and define 
$\gotb \colon W_{1-s}(H,V) \times W_{1-s}(H,V) \to \CC$ by
\[
\gotb(u,v)
= \int_0^\infty \Big( \langle u'(t),v'(t)\rangle_H
    +\mathcal E(u(t),v(t)) \Big) t^{2(1-s)}\,\frac{dt}{t}
.  \]
Define $j \colon W_{1-s}(H,V) \to H$ by $j(u) = u(0)$.
Then $\gotb$ is sectorial with vertex~$0$.
Let $B$ be the operator associated with $(\gotb,j)$.
Then $B = c_s A^s$, where $c_s = 2^{1-2s} \frac{\Gamma(1-s)}{\Gamma(s)}$.
\end{theorem}
\begin{proof}
It is easy to see that $\gotb$ is sectorial with vertex~$0$.
For all $n \in \NN$ define $\mathcal E_n \colon V\times V\to\CC$ by 
\[
\mathcal E_n(u,v) =\mathcal E(u,v)+\frac 1n\langle u,v\rangle_V.
\]
Then $\mathcal E_n$ is continuous and coercive.
Let $A_n$ be the $m$-sectorial operator in $H$ associated with $\mathcal E_n$.
Then $A_n$ is sectorial with vertex~$0$.
Moreover, $\lim_{n \to \infty} A_n = A$ in the strong resolvent sense
by \cite[Corollary 3.9]{AtE}.
Hence $\lim_{n \to \infty} A_n^s = A^s$ in the strong resolvent sense
by the representation formula \cite[(6) in Section~IX.11]{Yos}.
For all $n \in \NN$ define $\gotb_n \colon W_{1-s}(H,V) \times W_{1-s}(H,V) \to \CC$ by
\begin{eqnarray*}
\gotb_n(u,v)
&=& \int_0^\infty \Big( \langle u'(t),v'(t)\rangle_H
    +\mathcal E_n(u(t),v(t))\Big) t^{2(1-s)}\,\frac{dt}{t}  \\
&=&\gotb(u,v) + \frac 1n\int_0^\infty\langle u(t),v(t)\rangle_V t^{2(1-s)}\,\frac{dt}{t}.
\end{eqnarray*}
Then $\gotb_n$ is continuous and coercive.
Let $B_n$ be the operator associated with $(\gotb_n,j)$ as in Proposition~\ref{prop-52}.
Then $\lim_{n \to \infty} B_n = B$ in the strong resolvent sense
again by \cite[Corollary 3.9]{AtE}.
But $B_n = c_s A_n^s$ for all $n \in \NN$ by Theorem~\ref{theo-51}.
Taking the limit as $n \to \infty$ and using the uniqueness of the limit 
in the strong resolvent sense gives $B = c_s \, A^s$ as required.
\end{proof}

Adopt the notation and assumptions as in Theorem~\ref{ts602}.
We suppose from now on in addition that $\mathcal E$ is $H$-elliptic, that is 
there exists a constant $\mu > 0$ such that 
\begin{equation}
\RRe \mathcal E(u)+\|u\|_H^2\ge \mu\|u\|_V^2
\label{eS6;1}
\end{equation}
for all $u\in V$.
In this case we can give an explicit description of the operator $B$ and 
show that it is again a Dirichlet-to-Neumann map. 
For this we need quite some preparation.

Recall that if $X$ is a Banach space and $-\infty<\alpha<\beta<\infty$, then
$W^{1,1}((\alpha,\beta);X)\subset C([\alpha,\beta];X)$ and
\[
u(t)=u(\alpha)+\int_{\alpha}^tu'(r)\,dr
\]
for all $u\in W^{1,1}((\alpha,\beta);X)$ and $t\in [\alpha,\beta]$.
Conversely, if $x\in X$, $v\in L^1((\alpha,\beta);X)$ and 
$u \colon (\alpha,\beta) \to X$ is given by
$u(t)=x+\int_\alpha^tv(r)\,dr$, then
$u\in W^{1,1}((\alpha,\beta);X)$ and $u'=v$.
For all $-\infty\le a<b\le\infty$ and $p \in [1,\infty]$ we let
\begin{align*}
W_{\loc}^{1,p}((a,b);X)
   :=\{u\colon (a,b)\to X 
 : \;& u|_{(\alpha,\beta)}\in W^{1,p}((a,\beta);X) \\
 & \mbox{for all } \alpha,\beta \in \RR \mbox{ with } a < \alpha < \beta < b \}
\end{align*}
and if $a \neq - \infty$, then we define 
\[
L^p_\loc([a,b);X)
= \{ u \colon [a,b) \to X : u|_{[a,\beta)} \in L^p([a,\beta);X) 
       \mbox{ for all } \beta \in (a,b) \} 
.  \]
We always identify $u\in W_{\loc}^{1,p}((a,b);X)$ with its continuous representative.

Recall that $0<s<1$.
We define the space
\begin{eqnarray*}
\lefteqn{
W:=\{u\in C([0,\infty);H)\cap W_{\loc}^{1,2}((0,\infty);H)\cap L_{\loc}^2((0,\infty);V):
} \hspace*{30mm} \\*
& & \int_0^\infty \Big(\|u'(t)\|_H^2 + \RRe \mathcal E(u(t))\Big)
    t^{2(1-s)}\,\frac{dt}{t} <\infty \}.
\end{eqnarray*}
We provide $W$ with the norm
\[
\|u\|_W^2
= \|u(0)\|_H^2+\int_0^\infty\Big(\|u'(t)\|_H^2
   +\RRe \mathcal E(u(t))\Big) t^{2(1-s)}\,\frac{dt}{t}
.  \]
We first prove that $W$ is a Hilbert space.
For the proof we need two lemmas.

\begin{lemma}\label{lem-61-1}
Let $u\in W$.
Then $u'\in L_{\loc}^1([0,\infty);H)$.
Moreover, 
\begin{equation}
\int_0^t\|u'(r)\|_H\,dr
\leq \|u\|_W\cdot \left( \frac{1}{2s} t^{2s}\right)^{\frac 12}
\label{65-1}
\end{equation}
and 
\begin{align}\label{64-1}
\|u(t)\|_H^2
\le \frac{1}{s}\|u\|_W^2 (1 + t^{2s})
\end{align}
for all $t > 0$.
\end{lemma}

\begin{proof}
Let $u\in W$ and $t>0$.
Then
\begin{eqnarray*}
\int_0^t\|u'(r)\|_H\,dr
& \le &\left(\int_0^t \|u'(r)\|_H^2 r^{2(1-s)}\, \frac{dr}{r} \right)^{\frac 12}
   \left(\int_0^t r^{2s-1}\, dr\right)^{\frac 12} \\
& \le & \|u\|_W\cdot\left(\frac{1}{2s}t^{2s}\right)^{\frac 12}.
\end{eqnarray*}
This shows \eqref{65-1} and that $u'\in L_{\loc}^1([0,\infty);H)$.
If $a \in (0,t)$, then
\begin{align*}
u(t)=u(a)+\int_a^tu'(r)\,dr.
\end{align*}
Letting $a\downarrow 0$ gives
\begin{align}\label{63-1}
u(t)=u(0)+\int_0^tu'(r)\,dr.
\end{align}
Moreover, \eqref{63-1} and \eqref{65-1} give
\begin{eqnarray*}
\|u(t)\|_H^2
&\le& \frac 12\|u(0)\|_H^2+\frac 12\left(\int_0^t\|u'(r)\|_H\,dr\right)^{2}\\
&\le &\frac 12\|u(0)\|_H^2 + \frac{1}{4s} \|u\|_W^2  t^{2s}\\
&\le &\|u\|_W^2 \Big( 1 + \frac{1}{s} t^{2s} \Big)
\end{eqnarray*}
and \eqref{64-1} follows.
\end{proof}

Recall that $\mu$ is defined in \eqref{eS6;1}.

\begin{lemma}\label{lem-62-2}
If $u\in W$ and $T \geq 1$, then
\[
\mu\int_0^T\|u(t)\|_V^2 t^{2(1-s)}\,\frac{dt}{t}
\le \Big( 1+\frac{1}{s(1-s)} T^2\Big)\|u\|_W^2
.
\]
\end{lemma}

\begin{proof}
By $H$-ellipticity and \eqref{64-1} one estimates
\begin{eqnarray*}
\mu\int_0^T\|u(t)\|_V^2 t^{2(1-s)}\,\frac{dt}{t}
& \le&\int_0^T\RRe \mathcal E(u(t))  t^{2(1-s)}\,\frac{dt}{t}
   +\int_0^T\|u(t)\|_H^2 t^{2(1-s)}\,\frac{dt}{t}  \\
& \le &\|u\|_W^2 + \frac 1s \|u\|_W^2\int_0^T (1 + t^{2s} ) t^{1-2s}\,dt  \\
& = & \|u\|_W^2 \Big( 1 + \frac 1s \Big( \frac{1}{2(1-s)} T^{2(1-s)} + \frac{1}{2} T^2 \Big)\Big)
\end{eqnarray*}
and the lemma follows.
\end{proof}

\begin{proposition}\label{prop61}
The space $W$ is a Hilbert space.
\end{proposition}

\begin{proof}
Let $(u_n)_{n\in\NN}$ be a Cauchy sequence in~$W$.
Then $x:=\lim_{n\to\infty}u_n(0)$ exists in~$H$.
Moreover, there exists a $v\in L^2((0,\infty);H,t^{2(1-s)}\,\frac{dt}{t})$ such that 
$\lim_{n\to\infty} u_n' = v$ in $L^2((0,\infty);H,t^{2(1-s)}\,\frac{dt}{t})$.
It follows from \eqref{65-1} that $(u_n'|_{(0,T)})_{n\in\NN}$ is a Cauchy sequence in 
$L^1((0,T);H)$ for all $T > 0$.
Hence $v|_{(0,T)} \in L^1((0,T);H)$ and, moreover, 
$\lim_{n\to\infty} u_n'|_{(0,T)} = v|_{(0,T)}$ in $L^1((0,T);H)$.
In particular, $v \in L^1_\loc([0,\infty);H)$.
Define $u \colon [0,\infty) \to H$ by 
\begin{equation}
u(t) = x + \int_0^t v(r) \, dr.
\label{eprop61;1}
\end{equation}
Let $t > 0$.
By \eqref{63-1} we have
\[
u_n(t) = u_n(0) + \int_0^t u_n'(r) \, dr
\]
for all $n \in \NN$.
Hence 
\[
\lim_{n\to\infty} u_n(t)
= x + \int_0^t v(r) \, dr
= u(t)
\]
in $H$.

It follows from \eqref{eprop61;1} that 
$u \in C([0,\infty);H) \cap W_{\loc}^{1,2}((0,\infty);H)$ and $u' = v$.
Moreover, Lemma~\ref{lem-62-2} implies that 
$(u_n|_{(a,b)})_{n \in \NN}$ is also a Cauchy sequence in 
$L^2((a,b);V)$ for all $a,b \in (0,\infty)$ with $a < b$.
Since $\lim_{n \to \infty} u_n = u$ in $H$ pointwise, it follows that 
$u|_{(a,b)} \in L^2((a,b);V)$ and 
$\lim_{n\to\infty} u_n|_{(a,b)} = u|_{(a,b)}$ in $L^2((a,b);V)$.
Passing to a subsequence if necessary, we may assume that 
$\lim_{n\to\infty} u_n(t) = u(t)$ in $V$ for almost all $t \in (0,\infty)$.

Let $\varepsilon > 0$.
There exists an $N_0 \in \NN$ such that $\|u_n - u_m\|_W^2 \leq \varepsilon$
for all $n,m \in \NN$ with $n,m \geq N_0$.
Let $n \in \NN$ with $n \geq N_0$.
Then \cite[Lemma~VIII.3.14a]{Kat} and Fatou's lemma give
\begin{eqnarray*}
\lefteqn{
\|u_n(0) - u(0)\|_H^2 
   + \int_0^\infty \Big( \|u_n'(t) - u'(t)\|_H^2 + \RRe\mathcal E(u_n(t) - u(t)) \Big)
       t^{2(1-s)}\,\frac{dt}{t}
} \hspace*{0mm}  \\*
& \leq & \|u_n(0) - x\|_H^2 
   + \int_0^\infty \liminf_{m \to \infty}
            \Big( \|u_n'(t) - u_m'(t)\|_H^2 + \RRe\mathcal E(u_n(t) - u_m(t)) \Big)
       t^{2(1-s)}\,\frac{dt}{t}  \\
& \leq & \liminf_{m \to \infty} \|u_n - u_m\|_W^2
\leq \varepsilon
.
\end{eqnarray*}
Hence $u_n - u \in W$ and $u \in W$.
So $\|u_n - u\|_W^2 \leq \varepsilon$ for all $n \geq N_0$
and $\lim_{n \to \infty} u_n = u$ in $W$.
We have shown that the space $W$ is complete.
\end{proof}

We need one more lemma before we can give a Dirichlet-to-Neumann 
type description for the operator $c_s A^s$.

\begin{lemma}\label{lem-64-1}
The space $W_{1-s}(H,V)$ is dense in $W$.
\end{lemma}
\begin{proof}
Lemma \ref{lem-62-2} implies that $u\in W_{1-s}(H,V)$ for all $u\in W$ with compact support 
in $[0,\infty)$.
Let $u\in W$.
Let $\eta\in C_c^\infty[0,\infty)$ be such that $\one_{[0,1]}\le\eta\le \one_{[0,2]}$.
Let $n\in \NN$.
Define $\eta_n \in C_c^\infty[0,\infty)$ by $\eta_n(t):=\eta(\frac{t}{n})$.
Then $\one_{[0,n]}\le\eta_n\le \one_{[0,2n]}$.
Define $u_n:=\eta_nu \in W_{1-s}(H,V)$.
We shall show that $\sup_{n\in\NN}\|u_n\|_W<\infty$.
Obviously
\begin{align*}
\int_0^\infty\RRe\mathcal E(u_n(t)) t^{2(1-s)}\,\frac{dt}{t} \le\|u\|_W^2
\end{align*}
and
\begin{align*}
\int_0^\infty\|(\eta_nu')(t)\|_H^2 t^{2(1-s)}\,\frac{dt}{t} \le \|u\|_W^2.
\end{align*}
It remains to show that 
$\sup_{n \in \NN} \int_0^\infty\|(\eta_n'u)(t)\|_H^2 t^{2(1-s)}\,\frac{dt}{t} < \infty$.
Using \eqref{64-1} one estimates
\begin{eqnarray*}
\int_0^\infty\|(\eta_n'u)(t)\|_H^2 t^{2(1-s)}\,\frac{dt}{t}
& =&\frac{1}{n^2}\int_0^{2n}\Big(\eta'(\frac{t}{n})\Big)^2 \|u(t)\|_H^2 t^{1-2s}\,dt\\
& \le & \frac{\|\eta'\|_{\infty}^2 \|u\|_W^2}{n^2 s}
     \int_0^{2n} (1 + t^{2s}) t^{1-2s}\,dt\\
& =& \frac{\|\eta'\|_{\infty}^2 \|u\|_W^2}{n^2 s} 
    \Big( \frac{(2n)^{2-2s}}{2-2s} + 2n^2\Big)
.
\end{eqnarray*}
Hence $\sup_{n\in\NN}\|u_n\|_W<\infty$.

Passing to a subsequence if necessary, we have that there exists a $w \in W$ such that 
$\lim_{n\to\infty} u_n = w$ weakly in $W$.
Then \eqref{64-1} implies that $\lim_{n \to \infty} u_n(t) = w(t)$
in $H$ for almost all $t > 0$.
So $u = w \in W$.
We have shown that $u$ is in the weak closure of $W_{1-s}(H,V)$ in $W$.
Since $W_{1-s}(H,V)$ is convex, it is also in the strong closure.
Hence $W_{1-s}(H,V)$ is dense in $W$.
\end{proof}

Now we are able to show that the operator $B$ in Theorem~\ref{ts602} is 
a Dirichlet-to-Neumann map if $\mathcal E$ is $H$-elliptic.

\begin{theorem} \label{ts603}
Adopt the assumptions and notation as in Theorem~\ref{ts602}.
Moreover, assume that $\mathcal E$ is $H$-elliptic.
Define the form $\tilde \gotb \colon W \times W \to \CC$ by 
\[
\tilde \gotb(u,v)
= \int_0^\infty \Big( \langle u'(t),v'(t)\rangle_H
    +\mathcal E(u(t),v(t)) \Big) t^{2(1-s)}\,\frac{dt}{t}
.  \]
Let $x,y \in H$. 
Then the following assertions are equivalent.
\begin{tabeleq}
\item  \label{ts603-1}
$x\in D(A^s)$ and $c_sA^sx=y$.
\item  \label{ts603-2}
There exists a $u\in W$ such that $u(0)=x$ 
and $\tilde\gotb(u,v)=\langle y,v(0)\rangle_H$ for all $v\in W$.
\end{tabeleq}
\end{theorem}

\begin{proof}
Define $\tilde j \colon W \to H$ by $\tilde j(u) = u(0)$.
Then $\tilde \gotb$ is continuous and 
\[
\|u\|_W^2
\leq \RRe \tilde\gotb(u) + \|\tilde j(u)\|_H^2
\]
for all $u \in W$.
Moreover, $\tilde j$ is continuous and has dense image.
Obviously $\tilde \gotb$ and $\tilde j$ are extensions of $\gotb$ and~$j$, respectively.
In addition, $W$ is complete and $W_{1-s}(H,V)$ is dense in $W$ by Proposition~\ref{prop61}
and Lemma~\ref{lem-64-1}.
Hence by \cite[Proposition~3.3]{AtE} it follows that $B$ is the operator 
associated with $(\tilde \gotb,\tilde j)$ in the sense of Proposition~\ref{prop-52}.
Then the equivalence follows immediately from Theorem~\ref{ts602}
and the definition of the graph of $B$ in Proposition~\ref{prop-52}.
\end{proof}

Let $\tilde \gotb$ be as in Theorem~\ref{ts603}
and let $u \in W$ be as in Condition~\ref{ts603-2} in Theorem~\ref{ts603}.
Then $u\in W_{\loc}^{2,2}((0,\infty);V')$ and
\begin{align*}
u''(t)+\frac{1-2s}{t}u'(t)-\mathcal Au(t)=0
\mbox{ in } V' \mbox{ for a.e.\ } t \in (0,\infty),
\end{align*}
where $\mathcal A\colon V\to V'$ is given by 
$\langle \mathcal Aw,v\rangle_{V',V}=\mathcal E(w,v)$ for all $w,v\in V$.

\subsection*{Acknowledgements}
The first author is most grateful for a stimulating stay at the University of Puerto Rico, 
Rio Piedras Campus.
The second and third authors are most grateful for the hospitality extended to them 
during their fruitful stay at the University of Ulm.
The research of the third author is partially supported by the AFOSR Grant FA9550-15-1-0027.
Part of this work is supported by an
NZ-EU IRSES counterpart fund and the Marsden Fund Council from Government funding,
administered by the Royal Society of New Zealand.
Part of this work is supported by the
EU Marie Curie IRSES program, project `AOS', No.~318910.

\newpage

\end{document}